\documentclass{article}%

\usepackage{amsmath,amssymb,amsfonts}%
\usepackage{theorem}%
\usepackage{a4wide}%
\usepackage{color}%
\usepackage{dsfont}%
\usepackage{hyperref}%

\theoremstyle{change}%

\newtheorem{definition}{Definition:}[section]%
\newtheorem{theorem}[definition]{Theorem:}%
\newtheorem{proposition}[definition]{Proposition:}%
\newtheorem{lemma}[definition]{Lemma:}%
\newtheorem{corollary}[definition]{Corollary:}%
\newtheorem{conjecture}[definition]{Conjecture:}%
{\theorembodyfont{\rmfamily} \newtheorem{remark}[definition]{Remark:}}%
{\theorembodyfont{\rmfamily} }%

\newenvironment{proof}
  {{\bf Proof:}}
  {\qquad \hspace*{\fill} $\Box$}%

\newcommand{\AC}{\mathcal{A}}%
\newcommand{\BC}{\mathcal{B}}%
\newcommand{\FC}{\mathcal{F}}%
\newcommand{\KC}{\mathcal{K}}%
\newcommand{\MC}{\mathcal{M}}%
\newcommand{\PC}{\mathcal{P}}%
\newcommand{\SC}{\mathcal{S}}%
\newcommand{\UC}{\mathcal{U}}%
\newcommand{\VC}{\mathcal{V}}%
\newcommand{\WC}{\mathcal{W}}%
\newcommand{\tm}{\times}%
\newcommand{\K}{\mathbb{K}}%
\newcommand{\N}{\mathbb{N}}%
\renewcommand{\P}{\mathbb{P}}%
\newcommand{\R}{\mathbb{R}}%
\newcommand{\Z}{\mathbb{Z}}%
\newcommand{\rmd}{\mathrm{d}}%
\newcommand{\rme}{\mathrm{e}}%
\newcommand{\rmD}{\mathrm{D}}%
\newcommand{\SL}{\mathrm{SL}}%
\newcommand{\tr}{\mathrm{tr}}%
\newcommand{\rk}{\mathrm{rk}}%
\newcommand{\inv}{\mathrm{inv}}%
\newcommand{\diam}{\mathrm{diam}}%
\newcommand{\vol}{\mathrm{vol}}%
\newcommand{\inner}{\mathrm{int}}%
\newcommand{\cl}{\mathrm{cl}}%
\newcommand{\dist}{\mathrm{dist}}%
\newcommand{\tp}{\mathrm{top}}%
\newcommand{\ep}{\varepsilon}%
\newcommand{\unit}{\mathds{1}}%
\newcommand{\supp}{\mathrm{supp}}%
\newcommand{\Hom}{\mathrm{Hom}}%
\newcommand{\End}{\mathrm{End}}%
\newcommand{\Lip}{\mathrm{Lip}}%
\newcommand{\graph}{\mathrm{graph}}%
\newcommand{\length}{\mathrm{length}}%
\newcommand{\pr}{\mathrm{pr}}%

\begin{document}

\title{Invariance entropy for a class of partially hyperbolic sets}
\author{Christoph Kawan\thanks{Corresponding author; Faculty of Computer Science and Mathematics, University of Passau, 94032 Passau, Germany; phone: +49(0)851 5093363; e-mail: christoph.kawan@uni-passau.de} \and Adriano Da Silva\thanks{Imecc - Unicamp, Departamento de Matem\'atica, Rua S\'ergio Buarque de Holanda, 651, Cidade Universit\'aria Zeferino Vaz 13083-859, Campinas - SP, Brasil; e-mail: ajsilva@ime.unicamp.br}
}
\date{}
\maketitle%

\begin{abstract}
Invariance entropy is a measure for the smallest data rate in a noiseless digital channel above which a controller that only receives state information through this channel is able to render a given subset of the state space invariant. In this paper, we derive a lower bound on the invariance entropy for a class of partially hyperbolic sets. More precisely, we assume that $Q$ is a compact controlled invariant set of a control-affine system whose extended tangent bundle decomposes into two invariant subbundles $E^+$ and $E^{0-}$ with uniform expansion on $E^+$ and weak contraction on $E^{0-}$. Under the additional assumptions that $Q$ is isolated and that the $u$-fibers of $Q$ vary lower semicontinuously with the control $u$, we derive a lower bound on the invariance entropy of $Q$ in terms of relative topological pressure with respect to the unstable determinant. Under the assumption that this bound is tight, our result provides a first quantitative explanation for the fact that the invariance entropy does not only depend on the dynamical complexity on the set of interest.%
\end{abstract}



		
\section{Introduction}%

This paper is part of a major project with the aim to establish a hyperbolic theory of nonlinear control systems, to some extent analogous to the hyperbolic theory of dynamical systems originating from works of Poinc\'are, Hadamard, Smale, Anosov and many others. Insights of this theory are used in control at least since the work of Ott, Grebogi and Yorke \cite{OGY} on the subject known as `control of chaos', which uses properties of chaotic attractors to achieve stabilization with low energy use. Different from this approach, our work is grounded on the topological theory of Colonius-Kliemann \cite{CKl}, which aims at the understanding of the global controllability structure of a system using a dynamical systems view. More precisely, under appropriate assumptions, the transition map of the control system is extended to a skew-product flow (called \emph{control flow}), where it becomes a cocycle over a shift flow defined on the space of all admissible control functions. This construction is similar to the random dynamical systems view of stochastic ODEs. With the skew-product flow at hand, one can associate maximal sets of controllability with maximal topologically transitive sets of the skew-product, and give control-theoretic interpretations to other dynamical concepts such as chain transitive sets and supports of invariant measures.%

The paper \cite{CDu} was the first to examine uniformly hyperbolic structures on subsets of complete controllability (known as \emph{control sets}), using shadowing techniques to prove, e.g., that such sets behave robustly under small perturbations. In \cite{DK2}, it was shown that hyperbolic control sets exist for a large class of invariant control systems on flag manifolds of semisimple Lie groups, using extensively the semigroup theory developed by San Martin and coworkers \cite{ADS,SM1,SM2,SM3,SM4}. Moreover, it was shown that each maximal set of chain controllability (called a \emph{chain control set}) of such a system admits a continuous decomposition of its extended tangent bundle into three invariant subbundles $E^-$, $E^0$ and $E^+$ with uniform contraction on $E^-$ and uniform expansion on $E^+$. Further results on uniformly hyperbolic sets were obtained in \cite{DK1,DK3,Ka2}, and a survey of such results was provided in \cite{Ka3}.%

An important application of the hyperbolic theory of control systems is related to problems arising in networked control. A networked control system is a spatially distributed system whose components (sensors, controllers and actuators) can only communicate over a shared digital communication network. Examples can be found, e.g., in automated traffic systems, underwater communications for remotely controlled surveillance, rescue submarines and modern industrial systems which combine industrial production with information and communication technology. A fundamental problem in this context is to determine the minimal requirements on the communication network so that a given control objective can be achieved. In the simplest setup, this reduces to the question about the smallest channel capacity, above which a system can be stabilized. For linear control systems, it was shown by many authors under various different assumptions on the components of the system, that the infimal channel capacity equals the logarithm of the open-loop unstable determinant, which is known as the \emph{data-rate theorem} (see \cite{Ne2} for a survey). The fact that this value coincides with the (topological or measure-theoretic) entropy of the uncontrolled system led several researchers to the conclusion that methods from ergodic theory could be useful to study more advanced setups involving nonlinear dynamics, more complicated network topologies and also different control objectives, cf.~\cite{Col,Co2,DVa,Kaw,LHe,MPo,Nea,Sav,Yuk,YBa}.%

The paper at hand provides a contribution both to the theory of networked control and to the hyperbolic theory of control systems in that it extends a result obtained for the invariance entropy of uniformly hyperbolic control sets to a class of partially hyperbolic controlled invariant sets. Invariance entropy is a quantity that measures the smallest channel capacity of a noiseless digital channel above which a compact controlled invariant set can be made invariant by a controller that receives state information through this channel. This quantity, which is defined in a similar fashion as topological entropy of dynamical systems, was first introduced in \cite{Nea} for discrete-time systems under the name \emph{topological feedback entropy} (using an open-cover definition), and extended to continuous time in \cite{CKa} (using a spanning-sets definitions). As the results exposed in \cite{CKa,DK1,Kaw} show, invariance entropy is closely connected to dynamical quantities such as Lyapunov exponents and escape rates. In particular, in \cite{DK1} we provided a formula for the invariance entropy $h_{\inv}(Q)$ of a uniformly hyperbolic chain control set $Q$ without center bundle, namely%
\begin{equation}\label{eq_uhccs_formula}
  h_{\inv}(Q) = \inf_{(u,x) \in L(Q)}\limsup_{\tau\rightarrow\infty}\frac{1}{\tau}\log J^+\varphi_{\tau,u}(x)%
\end{equation}
with $J^+\varphi_{\tau,u}(x)$ denoting the unstable determinant of the spatial derivative of the transition map $\varphi(\tau,x,u) = \varphi_{\tau,u}(x)$ and%
\begin{equation*}
  L(Q) = \{ (u,x) \in \UC \tm M\ :\ \varphi(\R,x,u) \subset Q \},%
\end{equation*}
where $\UC$ is the set of admissible control functions. To prove the inequality `$\geq$' in \eqref{eq_uhccs_formula}, we only used ideas from dynamical systems related to the computation of escape rates, cf.~\cite{You}.%

In this paper, we extend the lower estimate in \eqref{eq_uhccs_formula} to a class of partially hyperbolic controlled invariant sets $Q$ with trivial dynamics (i.e., vanishing Lyapunov exponents) on the center bundle. Under the additional assumptions that $Q$ is isolated and that the $u$-fibers $Q(u) := \{ x\in Q\ :\ \varphi(\R,x,u) \subset Q\}$ depend lower semicontinuously on the control $u$ (with the appropriate topology on $\UC$), our main result shows that%
\begin{equation}\label{eq_ie_new_lb}
  h_{\inv}(Q) \geq -\sup_P P_{\tp}(\varphi_{|L(Q)};-\log J^+\varphi_{1,u}(x)).%
\end{equation}
Here the supremum is taken over all shift-invariant Borel probability measures $P$ on the base space $\UC$ of admissible control functions and $P_{\tp}(\cdot;\cdot)$ denotes the topological pressure of the bundle random dynamical system (briefly, a \emph{bundle RDS}) that is obtained from the control flow restricted to $L(Q)$, when $\UC$ is equipped with $P$. Observing that in the case of uniform hyperbolicity without center bundle, the topological pressure above reduces to%
\begin{equation*}
  P_{\tp}(\varphi_{|L(Q)};-\log J^+\varphi_{1,u}(x)) = -\inf_{\mu:\ (\pr_{\UC})_*\mu = P\atop \supp \mu \subset L(Q)} \int_{\UC} \log J^+\varphi_{1,u}(x) \rmd\mu(u,x),%
\end{equation*}
the infimum taken over all invariant probability measures of the control flow that project to $P$ and are supported on $L(Q)$, we are able to recover the lower bound in \eqref{eq_uhccs_formula} as a special case. Here we use that for a uniformly hyperbolic set without center bundle, the $u$-fibers of $Q$ are finite (see \cite{Ka2} for a proof), implying that the measure-theoretic entropy of the bundle RDS $\varphi_{|L(Q)}$ vanishes.%

As in the uniformly hyperbolic case, examples to which our result can be applied, are provided by invariant systems on flag manifolds on semisimple Lie groups. In the paper at hand, we only discuss a special case, namely systems on projective space induced by bilinear systems on $\R^d$, while the more general case is studied in \cite{DK4}.%

The paper is organized as follows. In Section \ref{sec_framework}, we precisely formulate our main result, recalling all the concepts involved. The proof is carried out in Section \ref{sec_proof} through a series of lemmas and propositions. Section \ref{sec_example} provides a class of examples and Section \ref{sec_interpret} discusses the contents of the result and relates it to other problems, e.g., submanifold stabilization. Finally, some technical results of independent interest, used in the proof, are collected in Section \ref{sec_appendix} (Appendix).%

\paragraph{Notation.} We write $\N = \{1,2,3,\ldots\}$ for the set of positive integers, $\Z$ for the set of all integers and $\R$ for the set of real numbers. We also write $\K_+ := \{ x\in \K : x \geq 0\}$ for $\K\in\{\Z,\R\}$. For $x>0$, $\log x$ denotes the natural logarithm of $x$. If $V,W$ are vector spaces, we write $\Hom(V,W)$ and $\End(V)$ for the spaces of homomorphisms from $V$ to $W$ and endomorphisms on $V$, respectively.%

If $M$ is a smooth manifold, we write $T_xM$ for the tangent space to $M$ at $x\in M$. If $f:M\rightarrow N$ is differentiable, $\rmD f(x):T_xM \rightarrow T_{f(x)}N$ denotes the derivative at $x\in M$. If $(M,g)$ is a Riemannian manifold, we write $|\cdot|$ for the norm in each tangent space $T_xM$, $d(\cdot,\cdot)$ for the geodesic distance and $\vol(\cdot)$ for the Riemannian volume measure on $M$, respectively. Moreover, $\exp_x$ denotes the Riemannian exponential function.%

If $(X,d)$ is a metric space, we write $B_r(x)$ for the open ball of radius $r>0$ centered at $x\in X$. If $\emptyset \neq A \subset X$ and $x\in X$, we write $\dist(x,A) := \inf_{a\in A}d(x,a)$. We let $N_{\ep}(A)$ denote the closed $\ep$-neighborhood of $A$, i.e.,%
\begin{equation*}
  N_{\ep}(A) := \left\{ x\in X\ :\ \dist(x,A) \leq \ep \right\}.%
\end{equation*}
Furthermore, $\unit_A$ stands for the indicator function of $A$ and we write $\inner A$, $\cl A$ and $\partial A$ for the interior, the closure and the boundary of $A$, respectively. We further define the diameter of $A$ by $\diam A := \sup_{x,y\in A}d(x,y)$. By $d_H(A,B)$ we denote the Hausdorff distance between two nonempty compact sets $A,B \subset X$. Finally, $\BC(X)$ denotes the Borel-$\sigma$-algebra on $X$.%

If $\mu$ is a Borel measure on $X$, we let $\supp\mu \subset X$ denote its support, and we write $f_*\mu$ for the push-forward of $\mu$ under a measurable map $f:X\rightarrow X$.%

\section{Assumptions and statement of the main result}\label{sec_framework}

We consider a control-affine system%
\begin{equation*}
  \dot{x}(t) = f_0(x(t)) + \sum_{i=1}^m u_i(t)f_i(x(t)),\quad u \in \UC,%
\end{equation*}
on a connected Riemannian manifold $M$ of dimension $d$. We assume that the vector fields $f_0,f_1,\ldots,f_m$ are of class $C^2$ and%
\begin{equation*}
  \UC = \left\{ u:\R \rightarrow \R^m\ :\ u \mbox{ measurable with } u(t) \in U \mbox{ a.e.} \right\},%
\end{equation*}
where $U \subset \R^m$ is an infinite compact and convex set. We write $\varphi(t,x,u) = \varphi_{t,u}(x)$ for the trajectory through $x\in M$, corresponding to the control function $u\in\UC$. Furthermore,
we denote by $\theta:\R \tm \UC \rightarrow \UC$, $(t,u) \mapsto \theta_tu := u(\cdot + t)$, the shift flow on $\UC$. With the weak$^*$-topology of $L^{\infty}(\R,\R^m) = L^1(\R,\R^m)^*$, $\UC$ becomes a compact metrizable space and $\theta$ a continuous flow, cf.~\cite{CKl}. We fix a metric $d_{\UC}$ on $\UC$, compatible with this topology.%

Assuming w.l.o.g.~that all trajectories are defined on $\R$, the maps $\theta$ and $\varphi$ define the control flow%
\begin{equation*}
  \Phi:\R \tm (\UC \tm M) \rightarrow \UC \tm M,\quad \Phi_t(u,x) = (\theta_tu,\varphi(t,x,u)),%
\end{equation*}
which is a continuous skew-product flow. In particular, $\varphi$ satisfies the cocycle property%
\begin{equation*}
  \varphi(t+s,x,u) = \varphi(t,\varphi(s,x,u),\theta_s u),\quad \forall t,s\in\R,\ x\in M,\ u\in\UC.%
\end{equation*}

From the assumptions it follows that $\varphi$ is twice differentiable in $x$ and that the first and second derivatives depend continuously on $(u,x)$, see \cite[Thm.~1.1]{Kaw}.%

Recall that the \emph{invariance entropy} is defined as follows. Let $\emptyset \neq K \subset Q \subset M$, $K$ compact, such that for every $x\in K$ there is $u\in\UC$ with $\varphi(\R_+,x,u)\subset Q$. For $\tau>0$, a set $\SC \subset \UC$ is called $(\tau,K,Q)$-spanning if for every $x\in K$ there is $u\in\SC$ with $\varphi([0,\tau],x,u) \subset Q$. Writing $r_{\inv}(\tau,K,Q)$ for the smallest cardinality of such a set, the invariance entropy of $(K,Q)$ is defined by%
\begin{equation*}
  h_{\inv}(K,Q) := \limsup_{\tau\rightarrow\infty}\frac{1}{\tau}\log r_{\inv}(\tau,K,Q) \in [0,\infty].%
\end{equation*}

Throughout the paper, we consider a compact set $Q \subset M$ with the following properties:%
\begin{enumerate}
\item[(P1)] $Q$ is \emph{all-time controlled invariant}, i.e., for every $x\in Q$ there exists $u\in\UC$ with $\varphi(\R,x,u) \subset Q$. We write%
\begin{equation*}
  L(Q) := \left\{ (u,x) \in \UC \tm M\ :\ \varphi(\R,x,u) \subset Q \right\}%
\end{equation*}
for the \emph{lift} of $Q$ to $\UC \tm M$ and note that $L(Q)$ is a compact $\Phi$-invariant set.%
\item[(P2)] For every $(u,x) \in L(Q)$ there exists a decomposition%
\begin{equation*}
  T_xM = E^+(u,x) \oplus E^{0-}(u,x)%
\end{equation*}
into two linear subspaces of constant dimensions such that%
\begin{equation*}
  \rmD\varphi_{t,u}(x)E^i(u,x) = E^i(\Phi_t(u,x)) \mbox{\quad for all\ } t \in \R,\ (u,x) \in L(Q) \mbox{\ and\ } i \in \{+,0-\},%
\end{equation*}
and the following holds: there exists a constant $\lambda>0$ such that for every $\ep>0$ there is $T>0$ with the following property. For all $(u,x)\in L(Q)$ and $t\geq T$,%
\begin{align*}
  |\rmD\varphi_{t,u}(x)v| \geq \rme^{\lambda t}|v| &\mbox{\quad if\ } v \in E^+(u,x),\\
  |\rmD\varphi_{t,u}(x)v| \leq \rme^{\ep t}|v| &\mbox{\quad if\ } v \in E^{0-}(u,x).%
\end{align*}
\item[(P3)] The set-valued map, defined by%
\begin{equation*}
  u \mapsto Q(u) := \left\{ x \in M\ :\ \varphi(\R,x,u) \subset Q \right\},%
\end{equation*}
from the space $\UC$ of control functions into the space of compact subsets of $Q$ is lower semicontinuous. We call $Q(u)$ the \emph{$u$-fiber} of $Q$.%
\item[(P4)] The set $Q$ is isolated, i.e., there exists a neighborhood $N \subset M$ of $Q$ such that $\varphi(\R,x,u) \subset N$ implies $(u,x) \in L(Q)$ for any $(u,x) \in \UC \tm M$.%
\end{enumerate}

Observe that in (P3) we can assume w.l.o.g.~that $Q(u) \neq \emptyset$ for all $u\in\UC$. Otherwise we replace $\UC$ with $\UC^* := \{ u\in\UC : Q(u) \neq \emptyset \}$ in the whole proof, using that $\UC^*$ is closed and $\theta$-invariant. In Section \ref{sec_example} we will provide examples of sets $Q$ satisfying all of the above assumptions.%

Under these assumptions, we provide a lower bound on the invariance entropy $h_{\inv}(K,Q)$ for compact sets $K \subset Q$ of positive volume.%

To explain our result, we introduce the set $\MC_{\Phi_1}(L(Q))$ of all Borel probability measures on $L(Q)$, invariant under the time-$1$ map $\Phi_1:\UC \tm M \rightarrow \UC \tm M$ of the control flow. If $\mu \in \MC_{\Phi_1}(L(Q))$ and $P = (\pr_{\UC})_*\mu$ is the projection of $\mu$ to $\UC$, then $P$ is a Borel probability measure on $\UC$, invariant under $\theta_1$. Observe that when $\UC$ is endowed with $P$, $\varphi$ can be regarded as a two-sided $C^2$-random dynamical system over the base $(\UC,\BC(\UC),P,\theta_1)$ (cf.~\cite[Def.~1.1.3]{Arn}). In this case, the measure-theoretic entropy $h_{\mu}(\varphi)$ is defined by%
\begin{equation}\label{eq_rds_mesent}
  h_{\mu}(\varphi) := \sup_{\AC}\lim_{n\rightarrow\infty}\frac{1}{n}\int_{\UC} H_{\mu_u}\Bigl(\bigvee_{i=0}^{n-1}\varphi_{i,u}^{-1}\AC\Bigr) \rmd P(u),%
\end{equation}
where the supremum is taken over all finite measurable partitions of $M$ and $\{\mu_u\}_{u\in\UC}$ is the $P$-almost everywhere defined family of sample measures on $M$ so that $\rmd \mu(u,x) = \rmd \mu_u(x) \rmd P(u)$ (see also \cite[Sec.~1.4]{Arn} and \cite{Bog,ZCa}).%

For $(u,x)\in L(Q)$ and $t\in\R$, we write%
\begin{equation}\label{eq_unstabledet}
  J^+\varphi_{t,u}(x) := \bigl|\det\rmD\varphi_{t,u}(x)_{|E^+(u,x)}:E^+(u,x) \rightarrow E^+(\Phi_t(u,x))\bigr|%
\end{equation}
and note that the map $(t,u,x) \mapsto \log J^+\varphi_{t,u}(x)$ is a real-valued continuous additive cocycle over the restriction of $\Phi$ to $L(Q)$, i.e.,%
\begin{equation*}
  \log J^+\varphi_{t+s,u}(x) = \log J^+\varphi_{t,u}(x) + \log J^+\varphi_{s,\theta_tu}(\varphi(t,x,u)),\quad \forall t,s\in\R,\ (u,x)\in L(Q).%
\end{equation*}
The continuity of this cocycle follows from the fact that the exponential separation of the subbundles $E^+$ and $E^{0-}$ implies the continuity of $E^+(\cdot,\cdot)$ and $E^{0-}(\cdot,\cdot)$ (see, e.g., \cite[Lem.~6.4]{Kaw}).%

Now we can formulate our result:%

\begin{theorem}\label{thm_mainresult}
Under the assumptions on $Q$ formulated above, for every compact set $K\subset Q$ with $\vol(K)>0$ the invariance entropy satisfies%
\begin{equation}\label{eq_ie_lb}
  h_{\inv}(K,Q) \geq \inf_{\mu\in\MC_{\Phi_1}(L(Q))}\left(\int_{L(Q)} \log J^+\varphi_{1,u}(x) \rmd\mu(u,x) - h_{\mu}(\varphi)\right).%
\end{equation}
\end{theorem}

\begin{remark}
Observing that for a fixed $\theta_1$-invariant measure $P$ on $\UC$, the quantity%
\begin{equation*}
  \sup_{\mu}\left( h_{\mu}(\varphi) - \int_{L(Q)} \log J^+\varphi_{1,u}(x) \rmd\mu(u,x) \right),%
\end{equation*}
the supremum taken over all $\Phi_1$-invariant measures $\mu$ with marginal $P$ on $\UC$, equals the topological pressure of the random dynamical system $\varphi$ over $(\UC,\BC(\UC),P,\theta_1)$ w.r.t.~the observable $-\log J^+\varphi_{1,u}(x)$ by the variational principle for bundle RDS \cite{ZCa}, we can also write \eqref{eq_ie_lb} as%
\begin{equation*}
  h_{\inv}(K,Q) \geq -\sup_P P_{\tp}(\varphi_{|L(Q)};-\log J^+\varphi_{1,u}(x)).%
\end{equation*}
\end{remark}

\section{Proof of the main result}\label{sec_proof}

We prove Theorem \ref{thm_mainresult} through a series of lemmas and propositions. The whole proof is subdivided into three subsections and proceeds along the following steps:%
\begin{enumerate}
\item[(1)] In Subsection \ref{subsec_uer}, we use the lower semicontinuity assumption (P3) to derive a first estimate on $h_{\inv}(K,Q)$ in terms of a quantity which can be regarded as a uniform escape rate from the $\ep$-neighborhoods of the fibers $Q(u)$, $u\in\UC$ (for arbitrary $\ep$). This estimate reads%
\begin{equation}\label{eq_firstbound}
  h_{\inv}(K,Q) \geq \limsup_{t\rightarrow\infty}\inf_{u\in\UC}-\frac{1}{t}\log\vol\left(Q(t,u,\ep)\right),%
\end{equation}
where $Q(t,u,\ep)$ is the set of all initial states $x\in M$ whose trajectories $\varphi(s,x,u)$ under the application of the control $u$ stay $\ep$-close to the shifted fiber $Q(\theta_su)$ at each time $s \in [0,t]$. The proof is a modification of \cite[Thm.~4.5]{DK1}, where the fibers were assumed to be singletons (and hence, the assumption of lower semicontinuity was trivially satisfied).%
\item[(2)] In Subsection \ref{subsec_interchange}, we use a version of the Bowen-Ruelle volume lemma \cite{You} for skew-products (proved in the Appendix) to show that the order of the $\limsup$ and the infimum in \eqref{eq_firstbound} can be interchanged under the limit for $\ep\downarrow0$, leading to%
\begin{equation}\label{eq_secondbound}
  h_{\inv}(K,Q) \geq \lim_{\ep\downarrow0}\inf_{u\in\UC}\limsup_{t\rightarrow\infty}-\frac{1}{t}\log\vol\left(Q(t,u,\ep)\right).%
\end{equation}
More precisely, we prove this for a time-discretized version of the control flow. The main idea in the proof of \eqref{eq_secondbound} is to show that the mapping $(t,u) \mapsto \log\vol(Q(t,u,\ep))$ can be approximated by subadditive cocycles over the shift on $\UC$. For continuous subadditive cocycles it is known that infima and limits in such expressions can be interchanged. Here we do not have continuity of the approximating cocycles, but continuity of $\log\vol(Q(t,u,\ep))$ with respect to $u$, which suffices to carry out the proof. For the volume lemma used here, the partial hyperbolicity assumption (P2) is essential. For the approximation result, also assumption (P4) is needed.%
\item[(3)] Finally, in Subsection \ref{subsec_rer} we show that the infimum over the controls $u\in\UC$ can be replaced by an infimum over the shift-invariant probability measures $P$ on $\UC$, more precisely:%
\begin{equation}\label{eq_thirdbound}
  h_{\inv}(K,Q) \geq \lim_{\ep\downarrow0}\inf_P\limsup_{t\rightarrow\infty}-\frac{1}{t}\int \log\vol\left(Q(t,u,\ep)\right)\rmd P(u).%
\end{equation}
The proof of \eqref{eq_thirdbound} is based on ideas used in subadditive ergodic optimization. Observing that the numbers
\begin{equation*}
  \limsup_{t\rightarrow\infty}-\frac{1}{t}\int \log\vol\left(Q(t,u,\ep)\right)\rmd P(u)%
\end{equation*}
are known as random escape rates in the theory of random dynamical systems, we can apply techniques from the standard proof of the variational principle for pressure of random dynamical systems to derive our main result from \eqref{eq_thirdbound}.%
\end{enumerate}

\subsection{An estimate in terms of a uniform escape rate}\label{subsec_uer}

We start with a simple observation.%

\begin{lemma}\label{lem_fiber_usc}
The map $u \mapsto Q(u)$ from $\UC$ into the space of nonempty compact subsets of $Q$ is upper semicontinuous. Hence, by property (P3), this map is continuous with respect to the Hausdorff metric.%
\end{lemma}

\begin{proof}
Since $Q$ is compact, it suffices to prove outer semicontinuity, i.e., for $u_n \rightarrow u$ in $\UC$ and $x_n \in Q(u_n)$ with $x_n \rightarrow x$ we have to show that $x\in Q(u)$. To this end, observe that the continuity of $\varphi(t,\cdot,\cdot)$ for every $t\in\R$ implies $\varphi(t,x,u) = \lim_{n\rightarrow\infty}\varphi(t,x_n,u_n)$. Since $\varphi(t,x_n,u_n) \in Q$ for every $n$ and $Q$ is closed, $\varphi(t,x,u) \in Q$ follows, implying that $x\in Q(u)$.%
\end{proof}

We will now prove a generalization of the lower estimate for the invariance entropy in \cite[Thm.~4.5]{DK1}. To this end, for arbitrary $t>0$, $u\in\UC$ and $\ep>0$ we introduce the set%
\begin{equation*}
  Q(t,u,\ep) := \{x\in M\ :\ \dist(\varphi(s,x,u),Q(\theta_su)) \leq \ep,\ 0 \leq s \leq t\} = \bigcap_{s\in[0,t]}\varphi_{s,u}^{-1}N_{\ep}(Q(\theta_su)).%
\end{equation*}

The following proposition only uses property (P3) to derive a first lower bound on $h_{\inv}(K,Q)$. In particular, it holds if $Q(u)$ is a singleton for every $u\in\UC$, since in this case $u \mapsto Q(u)$ is trivially lower semicontinuous. For this case, the proposition was proved in \cite{DK1}.%

\begin{proposition}\label{prop_general_lb}
For every compact subset $K\subset Q$ of positive volume and every $\ep>0$, the invariance entropy satisfies%
\begin{equation}\label{eq_ielb}
  h_{\inv}(K,Q) \geq \limsup_{t\rightarrow\infty}\inf_{u\in\UC}-\frac{1}{t}\log\vol\left(Q(t,u,\ep)\right).%
\end{equation}
\end{proposition}

\begin{proof}
The proof proceeds in two steps.%

\emph{Step 1.} For each $u\in\UC$ and $\tau>0$ we define the two sets%
\begin{equation*}
  Q(u,\tau) := \left\{ x\in M\ :\ \varphi(t,x,u) \in Q,\ \forall t \in [0,\tau] \right\}%
\end{equation*}
and%
\begin{equation*}
  Q^{\pm}(u,\tau) := \left\{ x \in M\ :\ \varphi(t,x,u) \in Q,\ \forall t \in [-\tau,\tau] \right\}.%
\end{equation*}
Precisely the same way as in \cite[Thm.~4.5]{DK1} (where the fibers $Q(u)$ were assumed to be singletons), we can prove that%
\begin{equation}\label{eq_qftchar}
  Q^{\pm}(u,\tau) = \bigcup_{u^*_{|[-\tau,\tau]} = u_{|[-\tau,\tau]}} Q(u^*).%
\end{equation}
Now let $\ep>0$ and $u_0\in\UC$. By continuity of $u \mapsto Q(u)$ in the Hausdorff metric, there exists a neighborhood $V$ of $u_0$ so that%
\begin{equation}\label{eq_small_hausdorff_distance}
  Q(v) \subset N_{\ep/2}(Q(u_0)) \mbox{ and } Q(u_0) \subset N_{\ep/2}(Q(v)) \mbox{\quad for all\ } v \in V.%
\end{equation}
By definition of the weak$^*$-topology, we can choose this neighborhood of the form%
\begin{equation*}
  V = \left\{ u\in\UC\ :\ \Bigl|\int_{\R}\langle u(t)-u_0(t),x_i(t) \rangle \rmd t\Bigr| < 1,\ 1 \leq i \leq k \right\}%
\end{equation*}
for some $x_1,\ldots,x_k \in L^1(\R,\R^m)$. We choose $\tau_0>0$ so that%
\begin{equation*}
  \int_{\R\backslash[-\tau_0,\tau_0]}|x_i(t)|\rmd t < \frac{1}{2\diam U},\quad i=1,\ldots,k.%
\end{equation*}
Now let $u \in V_{1/2} := \{u\in\UC\ :\ |\int_{\R} \langle u(t)-u_0(t),x_i(t)\rangle\rmd t| < 1/2,\ 1 \leq i \leq k\}$ and consider $u^* \in \UC$ with $u^*_{|[-\tau_0,\tau_0]} = u_{|[-\tau_0,\tau_0]}$. We have%
\begin{equation*}
  \left|\int_{\R}\langle u^*(t)-u(t),x_i(t)\rangle\rmd t\right| \leq \diam U \int_{\R\backslash[-\tau_0,\tau_0]}|x_i(t)|\rmd t < \frac{1}{2},%
\end{equation*}
implying that for $i = 1,\ldots,k$,%
\begin{align*}
  \left|\int_{\R}\langle u^*(t)-u_0(t),x_i(t)\rangle\rmd t\right| &\leq \left|\int_{\R}\langle u^*(t)-u(t),x_i(t)\rangle\rmd t\right|\\
	&\qquad + \left|\int_{\R} \langle u(t)-u_0(t),x_i(t)\rangle\rmd t\right| < \frac{1}{2} + \frac{1}{2} = 1.%
\end{align*}
Hence, $u^* \in V$ and thus $Q(u^*) \subset N_{\ep/2}(Q(u_0))$ by \eqref{eq_small_hausdorff_distance}. We also have $Q(u_0) \subset N_{\ep/2}(Q(u))$ by $V_{1/2} \subset V$ and \eqref{eq_small_hausdorff_distance}. Hence, $x\in Q(u^*)$ implies the existence of $y\in Q(u_0)$ with $d(x,y) \leq \ep/2$ and $z \in Q(u)$ with $d(y,z) \leq \ep/2$. This yields $x \in N_{\ep}(Q(u))$, so by \eqref{eq_qftchar},%
\begin{equation*}
  Q^{\pm}(u,\tau_0) \subset N_{\ep}(Q(u)) \mbox{\quad for all\ } u \in V_{1/2}.%
\end{equation*}
Letting $u_0$ range through $\UC$, the open sets $V_{1/2} = V_{1/2}(u_0)$ cover $\UC$. By compactness, we can choose a finite subcover. This implies the existence of $\tau_0>0$ with%
\begin{equation*}
  Q^{\pm}(u,\tau_0) \subset N_{\ep}(Q(u)) \mbox{\quad for all\ } u \in \UC.%
\end{equation*}

\emph{Step 2.} Now consider the invariance entropy. Since the desired estimate becomes trivial if $h_{\inv}(K,Q) = \infty$, we may assume that finite $(\tau,K,Q)$-spanning sets exist for all $\tau>0$.
If $\SC$ is a minimal $(\tau,K,Q)$-spanning set, then%
\begin{equation}\label{eq_ie_vol}
  K \subset \bigcup_{u\in\SC}Q(u,\tau).%
\end{equation}
The same arguments as used in \cite[Thm.~4.5]{DK1} show that for every $\ep>0$ there is $\tau>0$ so that for all $u\in\UC$ and $t>0$,%
\begin{equation*}
  \varphi_{\tau,u}(Q(u,2\tau+t)) \subset \bigcap_{s\in[0,t]}\varphi_{s,\theta_{\tau}u}^{-1}N_{\ep}(Q(\theta_{s+\tau}u)) = Q(t,\theta_{\tau}u,\ep).%
\end{equation*}
Writing $v := \theta_{\tau}u$, we thus obtain%
\begin{equation*}
  \vol(Q(u,2\tau+t)) \leq \vol\left(\varphi_{\tau,u}^{-1}Q(t,v,\ep)\right).%
\end{equation*}
Since $\varphi_{\tau,u}^{-1}$ does not influence the exponential behavior of the right-hand side as $t\rightarrow\infty$ and \eqref{eq_ie_vol} implies that for a minimal $(2\tau+t,K,Q)$-spanning set $\SC_{2\tau+t}$ we have%
\begin{equation*}
  0 < \vol(K) \leq r_{\inv}(2\tau+t,K,Q) \cdot \max_{u\in\SC_{2\tau+t}} \vol(Q(u,2\tau+t)),%
\end{equation*}
the desired estimate%
\begin{equation*}
  h_{\inv}(K,Q) \geq \limsup_{t\rightarrow\infty}\inf_{u\in\UC} -\frac{1}{2\tau+t} \log \vol\left(Q(t,u,\ep)\right) = \limsup_{t\rightarrow\infty}\inf_{u\in\UC}-\frac{1}{t}\log\vol\left(Q(t,u,\ep)\right)%
\end{equation*}
follows.%
\end{proof}

\subsection{Interchanging limit superior and infimum}\label{subsec_interchange}

Our goal in this subsection is to interchange the order of $\limsup$ and infimum in the estimate \eqref{eq_ielb}. To simplify matters, we will work with a time-discretized version of the control flow from now on. We start with the following lemma.%

\begin{lemma}\label{lem_contlem}
Fix $t>0$, $\ep>0$ and $r\in\N$. Consider $r+1$ equi-distributed points $0 = s_0 < s_1 < \cdots < s_r = t$ in $[0,t]$ and define%
\begin{equation*}
  Q(t,u,\ep;r) := \bigcap_{0\leq i \leq r}\varphi_{s_i,u}^{-1}N_{\ep}(Q(\theta_{s_i}u)).%
\end{equation*}
Then the function $u \mapsto \vol(Q(t,u,\ep;r))$ is continuous.%
\end{lemma}

\begin{proof}
By Lemma \ref{lem_volzero}, we have $\vol(\partial N_{\ep}(Q(u))) = 0$ for all $u \in \UC$, which we will use later. For brevity, we write $A_i(u) := N_{\ep}(Q(\theta_{s_i}u))$. Then%
\begin{equation*}
  \vol(Q(t,u,\ep;r)) = \int_{N_{\ep}(Q)} \unit_{A_0(u)}(x)\unit_{A_1(u)}(\varphi_{s_1,u}(x)) \cdots \unit_{A_r(u)}(\varphi_{s_r,u}(x)) \rmd x.%
\end{equation*}
Again for brevity, we write $Q^{\ep} := N_{\ep}(Q)$ and $f_i(u,x) := \unit_{A_i(u)}(\varphi_{s_i,u}(x))$. We fix $u\in\UC$ and prove continuity of $\vol(Q(t,u,\ep;r))$ at $u$. To this end, first observe that for arbitrary $\tilde{u}\in\UC$ we have%
\begin{align*} 
  &\left|\vol(Q(t,u,\ep;r)) - \vol(Q(t,\tilde{u},\ep;r))\right| \leq \allowdisplaybreaks\\
	&\qquad \Bigl|\int_{Q^{\ep}} \left( f_0(u,x) f_1(u,x) \cdots f_r(u,x) - f_0(\tilde{u},x) f_1(u,x) \cdots f_r(u,x) \right)\rmd x\Bigr|\allowdisplaybreaks\\
  &\qquad + \Bigl|\int_{Q^{\ep}} \left( f_0(\tilde{u},x) f_1(u,x) \cdots f_r(u,x) - f_0(\tilde{u},x) f_1(\tilde{u},x) f_2(u,x) \cdots f_r(u,x) \right) \rmd x\Bigr|\allowdisplaybreaks\\
	&\qquad + \cdots \allowdisplaybreaks\\
	&\qquad + \Bigl|\int_{Q^{\ep}} \left( f_0(\tilde{u},x) \cdots f_{r-1}(\tilde{u},x)f_r(u,x) - f_0(\tilde{u},x) \cdots f_r(\tilde{u},x) \right)\rmd x\Bigr|.%
\end{align*}
Now observe that we can estimate the above by%
\begin{equation*}
  \sum_{i=0}^r \int_{Q^{\ep}} |f_i(u,x) - f_i(\tilde{u},x)|\rmd x.%
\end{equation*}
The integral%
\begin{equation*}
  \int_{Q^{\ep}} |f_i(u,x) - f_i(\tilde{u},x)| \rmd x = \int_{Q^{\ep}}\left| \unit_{A_i(u)}(\varphi_{s_i,u}(x)) - \unit_{A_i(\tilde{u})}(\varphi_{s_i,\tilde{u}}(x)) \right| \rmd x%
\end{equation*}
is not larger than the volume of the symmetric set difference%
\begin{equation}\label{eq_symmdiff}
  \left[\varphi_{s_i,u}^{-1}(A_i(u)) \backslash \varphi_{s_i,\tilde{u}}^{-1}(A_i(\tilde{u}))\right] \cup \left[\varphi_{s_i,\tilde{u}}^{-1}(A_i(\tilde{u})) \backslash \varphi_{s_i,u}^{-1}(A_i(u))\right].%
\end{equation}
We show that the volumes of these two sets become arbitrarily small as $\tilde{u}\rightarrow u$:%
\begin{enumerate}
\item[(i)] We write the first term in \eqref{eq_symmdiff} as%
\begin{equation*}
  \varphi_{s_i,u}^{-1}(A_i(u)) \backslash \varphi_{s_i,\tilde{u}}^{-1}(A_i(\tilde{u})) = \varphi_{s_i,u}^{-1}\left(A_i(u) \backslash \varphi_{s_i,u}(\varphi_{s_i,\tilde{u}}^{-1}(A_i(\tilde{u})))\right).%
\end{equation*}
Since $u$ is fixed, it suffices to show that the volume of $A_i(u) \backslash \varphi_{s_i,u}(\varphi_{s_i,\tilde{u}}^{-1}(A_i(\tilde{u})))$ becomes small as the distance $d_{\UC}(\tilde{u},u)$ becomes small. Using the notation $I_{\rho}(B) := \{ x\in\inner B\ :\ \dist(x,\partial B) \geq \rho \}$ for any subset $B\subset M$, it is enough to show that%
\begin{equation*}
	A_i(u) \backslash \varphi_{s_i,u}(\varphi_{s_i,\tilde{u}}^{-1}(A_i(\tilde{u}))) \subset A_i(u) \backslash I_{\rho}(A_i(u))%
\end{equation*}
for arbitrarily small $\rho$ as $\tilde{u}\rightarrow u$. Here we use continuity of the measure and $\vol(\partial A_i(u)) = 0$. The above inclusion is implied by%
\begin{equation*}
  \varphi_{s_i,\tilde{u}} \circ \varphi_{s_i,u}^{-1}(I_{\rho}(A_i(u))) \subset A_i(\tilde{u}).%
\end{equation*}
Take $x \in I_{\rho}(A_i(u))$ and let $y \in Q(\theta_{s_i}u)$ be a point that minimizes the distance $d(x,y)$, i.e., $d(x,y) = \dist(x,Q(\theta_{s_i}u))$. Let $\tilde{y} \in Q(\theta_{s_i}\tilde{u})$ be chosen so that $d(y,\tilde{y}) \leq d_H(Q(\theta_{s_i}\tilde{u}),Q(\theta_{s_i}u))$. Then%
\begin{equation*}
  d(\varphi_{s_i,\tilde{u}} \circ \varphi_{s_i,u}^{-1}(x),\tilde{y}) \leq d(\varphi_{s_i,\tilde{u}} \circ \varphi_{s_i,u}^{-1}(x),x) + d(x,y) + d_H(Q(\theta_{s_i}\tilde{u}),Q(\theta_{s_i}u)).%
\end{equation*}
If we can show that this sum becomes smaller than $\ep$ (independently of the choice of $x$) if $d_{\UC}(\tilde{u},u)$ is chosen small enough, we are done. The third term becomes small by continuity of $Q(\cdot)$ and $\theta$. The first term becomes small by the continuity properties of $\varphi$. Indeed, $\varphi(s_i,\cdot,\cdot)$ is uniformly continuous on an appropriately chosen compact set, showing that $d(\varphi_{s_i,\tilde{u}}(\varphi_{s_i,u}^{-1}(x)),\varphi_{s_i,u}(\varphi_{s_i,u}^{-1}(x))) \rightarrow 0$ as $\tilde{u} \rightarrow u$, independently of the choice of $x$. Now $x \in I_{\rho}(A_i(u))$ implies that the second term is smaller than and uniformly bounded away from $\ep$. This implies the result.%
\item[(ii)] Consider now the second term in \eqref{eq_symmdiff}. Writing%
\begin{equation*}
  \varphi_{s_i,\tilde{u}}^{-1}(A_i(\tilde{u})) \backslash \varphi_{s_i,u}^{-1}(A_i(u)) = \varphi_{s_i,u}^{-1}\bigl(\varphi_{s_i,u} \circ \varphi_{s_i,\tilde{u}}^{-1} (A_i(\tilde{u}))\backslash A_i(u) \bigr),%
\end{equation*}
we see that it suffices to prove that the volume of $\varphi_{s_i,u} \circ \varphi_{s_i,\tilde{u}}^{-1} (A_i(\tilde{u}))\backslash A_i(u)$ tends to zero as $\tilde{u}\rightarrow u$. By the continuity properties of $\varphi$ it follows that for any $\rho>0$ we have%
\begin{equation*}
  \varphi_{s_i,u} \circ \varphi_{s_i,\tilde{u}}^{-1}(A_i(\tilde{u})) \subset N_{\rho}(A_i(\tilde{u})),%
\end{equation*}
provided that $d_{\UC}(\tilde{u},u)$ is sufficiently small. Hence,%
\begin{equation*}
  \varphi_{s_i,u} \circ \varphi_{s_i,\tilde{u}}^{-1} (A_i(\tilde{u}))\backslash A_i(u) \subset N_{\rho+\ep}(Q(\theta_{s_i}\tilde{u})) \backslash N_{\ep}(Q(\theta_{s_i}u)).%
\end{equation*}
From the Hausdorff convergence $Q(\theta_{s_i}\tilde{u}) \rightarrow Q(\theta_{s_i}u)$ it follows that $N_{\rho+\ep}(Q(\theta_{s_i}\tilde{u})) \subset N_{2\rho+\ep}(Q(\theta_{s_i}u))$ for $d_{\UC}(\tilde{u},u)$ sufficiently small. Hence,%
\begin{equation*}
  \varphi_{s_i,u} \circ \varphi_{s_i,\tilde{u}}^{-1} (A_i(\tilde{u}))\backslash A_i(u) \subset N_{2\rho+\ep}(Q(\theta_{s_i}u)) \backslash N_{\ep}(Q(\theta_{s_i}u)).%
\end{equation*}
Now the volume of this set certainly tends to zero as $\rho\rightarrow0$ (by continuity of the measure).%
\end{enumerate}
We have proved that $u \mapsto \vol(Q(t,u,\ep;r))$ is continuous for any $r\in\N$.%
\end{proof}

\begin{remark}
It is easy to show, as a corollary, that also the function $u \mapsto \vol(Q(t,u,\ep))$ is continuous for fixed $t,\ep$. However, we will not need this for our proof.%
\end{remark}

Since $Q(t,u,\ep) \subset Q(t,u,\ep;r(t))$ for any $r(t)\in\N$, Proposition \ref{prop_general_lb} immediately implies the following corollary.

\begin{corollary}
For every $\ep>0$, the invariance entropy satisfies%
\begin{equation}\label{eq_discretized_lb}
  h_{\inv}(K,Q) \geq \limsup_{\N\ni n\rightarrow\infty}\inf_{u\in\UC}-\frac{1}{n}\log\vol\left(Q(n-1,u,\ep;n-1)\right).%
\end{equation}
\end{corollary}

We denote the set $Q(n-1,u,\ep;n-1)$ by $Q_{\rmd}(n,u,\ep)$ and study the properties of the functions%
\begin{equation*}
  v^{\ep}:\N \tm \UC \rightarrow \R,\quad v^{\ep}_n(u) := \log\vol(Q_{\rmd}(n,u,\ep)),\quad \ep>0.%
\end{equation*}
From Lemma \ref{lem_contlem} we know that each $v^{\ep}$ is continuous. We would like to interchange the order of the infimum and the limsup in \eqref{eq_discretized_lb}. This would be possible if $v^{\ep}$ was a subadditive cocycle over $\theta_1:\UC\rightarrow\UC$. In general, this is not the case. However, we can approximate $v^{\ep}$ by subadditive cocycles, which we will describe in the following.%

First observe that for any $n\in\N$, $u\in\UC$ and $\ep>0$, the set%
\begin{equation*}
  Q_{\rmd}(n,u,\ep) = \left\{ x\in M\ :\ \dist(\varphi_{j,u}(x),Q(\theta_ju)) \leq \ep,\ 0 \leq j < n \right\}%
\end{equation*}
is closed and bounded, hence compact if $\ep$ is sufficiently small.%

In the following, we will make use of a special version of the Bowen-Ruelle volume lemma. To formulate this lemma, we introduce for $n\in\N$, $\delta>0$ and $(u,x) \in \UC \tm M$ the set%
\begin{equation}\label{eq_def_bowenball}
  B^{n,u}_{\delta}(x) := \left\{ y\in M\ :\ d(\varphi(j,x,u),\varphi(j,y,u)) \leq \delta,\ 0 \leq j < n \right\},%
\end{equation}
which is called a \emph{Bowen-ball of order $n$ and radius $\delta$ centered at $x$}. We say that a set $F \subset M$ $(n,u,\delta)$-spans another set $K \subset M$ if the Bowen-balls $B^{n,u}_{\delta}(x)$, $x\in F$, form a cover of $K$. A set $E \subset M$ is $(n,u,\delta)$-separated if for all $x,y\in E$ with $x\neq y$, $d(\varphi(j,x,u),\varphi(j,y,u)) > \delta$ holds for some $j \in \{0,1,\ldots,n-1\}$.%

The proof of the following lemma can be found in Subsection \ref{subsec_volume_lemma} of the Appendix.%

\begin{lemma}\label{lem_volume}
For all sufficiently small $\ep>0$ and $\delta_0>0$ the following holds. If $u\in\UC$, $n\in\N$, $x \in Q_{\rmd}(n,u,\ep)$ and $\delta \in (0,\delta_0]$, then%
\begin{equation*}
  C_{\delta}^{-1}\rme^{-n\zeta}J^+\varphi_{n,u}(x)^{-1} \leq \vol(B^{n,u}_{\delta}(x)) \leq C_{\delta}\rme^{n\zeta}J^+\varphi_{n,u}(x)^{-1},%
\end{equation*}
where $C_{\delta}>0$ is a constant only depending on $\delta$, $\zeta>0$ is some small number, and $(u,x) \mapsto J^+\varphi_{n,u}(x)$, for each $n\in\N$, is a continuous extension of the map \eqref{eq_unstabledet} to a neighborhood of $L(Q)$, satisfying%
\begin{equation*}
  J^+\varphi_{n+m,u}(x) = J^+\varphi_{n,u}(x) \cdot J^+\varphi_{m,\theta_nu}(\varphi(n,x,u)),%
\end{equation*}
whenever both sides of this equation are defined. Moreover, $\zeta$ can be chosen arbitrarily small in dependence on $\ep$ and $\delta_0$.%
\end{lemma}

Let $\AC = (\AC_j)_{j=0}^{\infty}$ be a sequence so that $\AC_j$ is an open cover of the compact set $N_{\ep}(Q(\theta_ju))$ (i.e., a collection of open sets in $M$ whose union contains $N_{\ep}(Q(\theta_ju))$). We write%
\begin{equation*}
  \AC^n := \bigvee_{j=0}^{n-1}\varphi_{j,u}^{-1}(\AC_j),\quad n \in \N.%
\end{equation*}
This is the collection of all sets of the form%
\begin{equation*}
  A_0 \cap \varphi_{1,u}^{-1}(A_1) \cap \ldots \cap \varphi_{n-1,u}^{-1}(A_{n-1}),\quad A_j \in \AC_j.%
\end{equation*}
Observe that $\AC^n$ is an open cover of the compact set $Q_{\rmd}(n,u,\ep)$. We define%
\begin{equation}\label{eq_defwcocycle}
  w^{\ep,\AC}_n(u) := \log\inf\left\{ \sum_{A\in \alpha}\sup_{x\in A}J^+\varphi_{n,u}(x)^{-1} : \alpha \mbox{ is a finite subcover of } \AC^n \mbox{ for } Q_{\rmd}(n,u,\ep) \right\}.%
\end{equation}
Moreover, we write $\AC(n)$ for the shifted sequence $\AC_n,\AC_{n+1},\AC_{n+2},\ldots$.%

Now let $\alpha$ be a finite subcover of $\AC^n$ for $Q_{\rmd}(n,u,\ep)$ and $\beta$ a finite subcover of $\AC(n)^m$ for $Q_{\rmd}(m,\theta_nu,\ep)$. Then%
\begin{align*}
 \sum_{C \in \alpha \vee \varphi_{n,u}^{-1}(\beta)}\sup_{z\in C}J^+\varphi_{n+m,u}(z)^{-1} &= \sum_{C \in \alpha \vee \varphi_{n,u}^{-1}(\beta)}\sup_{z\in C}\Bigl[J^+\varphi_{n,u}(z)^{-1} \cdot J^+\varphi_{m,\theta_nu}(\varphi_{n,u}(z))^{-1}\Bigr]\\
 &\leq \sum_{(A,B) \in \alpha \tm \beta}\Bigl[\sup_{x\in A}J^+\varphi_{n,u}(x)^{-1}\Bigr] \cdot \Bigl[\sup_{y\in B}J^+\varphi_{m,\theta_nu}(y)^{-1}\Bigr]\\
 &= \Bigl[\sum_{A\in\alpha}\sup_{x\in A}J^+\varphi_{n,u}(x)^{-1}\Bigr] \cdot \Bigl[\sum_{B\in\beta}\sup_{y\in B}J^+\varphi_{m,\theta_nu}(y)^{-1}\Bigr].%
\end{align*}                                                                               
Hence, if we take subcovers $\alpha$ and $\beta$ so that the corresponding sums are close to the infimum, we see that%
\begin{equation}\label{eq_subadd}
  w^{\ep,\AC}_{n + m}(u) \leq w^{\ep,\AC}_n(u) + w^{\ep,\AC(n)}_m(\theta_nu).%
\end{equation}
Here we use that $\alpha \vee \varphi_{n,u}^{-1}(\beta)$ is a subcover of $\AC^{n+m}$ for $Q_{\rmd}(n+m,u,\ep)$.%

For small $\delta>0$, we define%
\begin{equation*}
  w^{\ep,\delta}_n(u) := w^{\ep,\AC(u)}_n(u),%
\end{equation*}
where $\AC(u)$ is the unique sequence so that $\AC(u)_j$ consists of all $(\delta/2)$-balls in $M$ that have a nonempty intersection with $N_{\ep}(Q(\theta_ju))$. Observe that then $J^+\varphi_{n,u}(x)$ in \eqref{eq_defwcocycle} is defined.%

\begin{proposition}\label{prop_wprops}
The functions $w^{\ep,\delta}$ have the following properties:%
\begin{enumerate}
\item[(i)] For all $n,m\in\N$ and $u\in\UC$,%
\begin{equation*}
  w^{\ep,\delta}_{n+m}(u) \leq w^{\ep,\delta}_n(u) + w^{\ep,\delta}_m(\theta_nu),%
\end{equation*}
i.e., $w^{\ep,\delta}$ is a subadditive cocycle over $\theta_1:\UC\rightarrow\UC$.%
\item[(ii)] For all $n\in\N$ and $u\in\UC$ it holds that%
\begin{equation*}
  v^{\ep}_n(u) \leq n\zeta + \log C_{\delta} + w^{\ep,\delta/2}_n(u),%
\end{equation*}
where $C_{\delta}$ and $\zeta$ are the constants in the Bowen-Ruelle volume lemma.%
\item[(iii)] For every $\alpha>0$ there exists $\delta>0$ (independent of $\ep$) so that%
\begin{equation*}
  w^{\ep,\delta/2}_n(u) - n(\alpha + \zeta) - \log C_{\delta/4} \leq v_n^{\ep+\delta/4}(u)%
\end{equation*}
for all $u \in \UC$ and $n\in\N$.%
\item[(iv)] For every sufficiently small $\delta>0$ there exists $N\in\N$ such that for all $u\in\UC$ and $n>2N$,%
\begin{equation*}
  v_n^{\ep+\delta}(u) \leq C(\ep) + v_{n-2N}^{\ep}(\theta_Nu)%
\end{equation*}
with a constant $C(\ep) \in \R$.%
\end{enumerate}
\end{proposition}

\begin{proof}
(i) This is immediately clear from \eqref{eq_subadd}.%

(ii) Let $\AC := \AC(u)$ and let $E \subset Q_{\rmd}(n,u,\ep)$ be a maximal $(n,u,\delta)$-separated set. Then each member of $\AC^n$ contains at most one element of $E$. Indeed, if there were two such elements $x_1$ and $x_2$, then%
\begin{equation*}
  d(\varphi(j,x_1,u),\varphi(j,x_2,u)) < \delta \mbox{\quad for\ } j = 0,1,\ldots,n-1,%
\end{equation*}
a contradiction. Hence, for every finite subcover $\alpha$ of $\AC^n$,%
\begin{equation*}
  \sum_{x\in E}J^+\varphi_{n,u}(x)^{-1} \leq \sum_{A\in\alpha}\sup_{x\in A}J^+\varphi_{n,u}(x)^{-1}.%
\end{equation*}
Since every maximal $(n,u,\delta)$-separated set is also $(n,u,\delta)$-spanning, Lemma \ref{lem_volume} implies%
\begin{equation*}
  v^{\ep}_n(u) = \log\vol(Q_{\rmd}(n,u,\ep)) \leq \log\sum_{x\in E} C_{\delta} \rme^{n\zeta} J^+\varphi_{n,u}(x)^{-1} \leq n\zeta + \log C_{\delta} + w^{\ep,\AC}_n(u).%
\end{equation*}

(iii) For the given number $\alpha>0$ choose $\delta>0$ small enough so that%
\begin{equation}\label{eq_alpha_est}
  \frac{J^+\varphi_{1,u}(x_1)}{J^+\varphi_{1,u}(x_2)} \leq \rme^{\alpha}%
\end{equation}
for all $x_1,x_2$ in a compact neighborhood of $Q$ satisfying $d(x_1,x_2)\leq\delta$ and all $u\in\UC$. This is possible by the uniform continuity of $(u,x) \mapsto J^+\varphi_{1,u}(x)$ on compact sets.%

Let $\AC := \AC(u)$ and consider a finite $(n,u,\delta/2)$-spanning set $F$ for $Q_{\rmd}(n,u,\ep)$, contained in $Q_{\rmd}(n,u,\ep)$. For each $z\in F$ consider sets $A_j(z) \in \AC_j$ so that $B_{\delta/2}(\varphi(j,z,u)) = A_j(z)$ for $j=0,1,\ldots,n-1$. Let%
\begin{equation*}
  C(z) := \bigcap_{j=0}^{n-1}\varphi_{j,u}^{-1}(A_j(z)) \in \AC^n.%
\end{equation*}
The definition of $C(z)$ together with \eqref{eq_alpha_est} implies%
\begin{equation*}
  \sup_{x \in C(z)}J^+\varphi_{n,u}(x)^{-1} \leq \rme^{n\alpha} \cdot J^+\varphi_{n,u}(z)^{-1}.%
\end{equation*}
Since the sets $C(z)$, $z \in F$, form a finite subcover of $\AC^n$ for $Q_{\rmd}(n,u,\ep)$,%
\begin{equation*}
  w^{\ep,\AC}_n(u) \leq n\alpha + \log\sum_{z \in F}J^+\varphi_{n,u}(z)^{-1}.%
\end{equation*}
Since a maximal $(n,u,\delta/2)$-separated set is also $(n,u,\delta/2)$-spanning and the corresponding Bowen-balls of radius $\delta/4$ are disjoint and contained in $Q_{\rmd}(n,u,\ep + \delta/4)$, Lemma \ref{lem_volume} implies%
\begin{equation*}
  w^{\ep,\AC}_n(u) \leq n(\alpha + \zeta) + \log C_{\delta/4} + v_n^{\ep+\delta/4}(u).%
\end{equation*}

(iv) We claim that for sufficiently small $\delta>0$ the following holds: for every sufficiently small $\ep>0$ there exists $N\in\N$ so that for all $u \in \UC$ and $x \in M$ it holds that%
\begin{equation}\label{eq_claimiv}
  \max_{-N<j<N}\dist(\varphi(j,x,u),Q(\theta_ju)) \leq \ep + \delta \quad \Rightarrow \quad \dist(x,Q(u)) < \ep.%
\end{equation}
To prove this claim, suppose to the contrary that there exists $\ep>0$ so that for every $N\in\N$ there are $u_N\in\UC$ and $x_N\in M$ with%
\begin{equation*}
  \dist(\varphi(j,x_N,u_N),Q(\theta_ju_N)) \leq \ep + \delta \mbox{ for } -N<j<N \mbox{\quad and\quad} \dist(x_N,Q(u_N)) \geq \ep.%
\end{equation*}
By compactness of $\UC$ we may assume $u_N \rightarrow u \in \UC$ and by compactness of small closed neighborhoods of $Q$, we may assume $x_N \rightarrow x\in M$. 

For arbitrary $j\in\Z$ we have $\dist(\varphi(j,x_N,u_N),Q(\theta_j u_N)) \leq \ep + \delta$ whenever $N > |j|$. Since $\varphi(j,\cdot,\cdot)$, $Q(\cdot)$ and $\dist(\cdot,\cdot)$ are continuous functions, this implies $\dist(\varphi(j,x,u),Q(\theta_ju)) \leq \ep + \delta$ for all $j\in\Z$. For the same reason, $\dist(x,Q(u)) \geq \ep$. By continuity of $\varphi$ it follows that $\varphi(\R,x,u)$ is contained in some $\tilde{\ep}$-neighborhood of $Q$, where $\tilde{\ep}$ tends to zero as $\ep + \delta$ tends to zero. Hence, if $\ep$ and $\delta$ are small enough so that $N_{\tilde{\ep}}(Q)$ is contained in an isolating neighborhood of $Q$, then Property (P4) implies $(u,x)\in L(Q)$, in contradiction to $\dist(x,Q(u)) \geq \ep$.%

Now choose $N$ according to \eqref{eq_claimiv} for given $\ep,\delta>0$ and pick $x \in Q_{\rmd}(n,u,\ep+\delta)$ for some $n>2N$. We want to show that%
\begin{equation*}
  \varphi_{N,u}(x) \in Q_{\rmd}(n-2N,\theta_Nu,\ep),%
\end{equation*}
which is equivalent to $\dist(\varphi(N+i,x,u),Q(\theta_{N+i}u)) \leq \ep$ for $0 \leq i < n-2N$. To show this, let $x_i := \varphi(N+i,x,u)$ for $0 \leq i < n-2N$ and observe that%
\begin{equation*}
  \dist(\varphi(j,x_i,\theta_{N+i}u),Q(\theta_j \theta_{N+i}u)) = \dist(\varphi(N+i+j,x,u),Q(\theta_{N+i+j}u)) \leq \ep + \delta%
\end{equation*}
for $-N < j < N$, since $0 < N+i+j < N + (n-2N) + (N-1) = n - 1$. Hence, $\dist(x_i,Q(\theta_{N+i}u)) < \ep$ for $0 \leq i < n-2N$, implying $\varphi_{N,u}(x) \in Q_{\rmd}(n-2N,\theta_Nu,\ep)$. It follows that $\varphi_{N,u}(Q_{\rmd}(n,u,\ep+\delta)) \subset Q_{\rmd}(n-2N,\theta_Nu,\ep)$, so%
\begin{equation*}
  v_n^{\ep+\delta}(u) \leq C + v_{n-2N}^{\ep}(\theta_Nu)%
\end{equation*}
for the constant $C := \max_{(u,x)}\log|\det\rmD\varphi_{N,u}^{-1}(x)|$, the maximum taken over $\UC \tm N_{\ep}(Q)$.%
\end{proof}

\begin{remark}
Note that it is totally unclear if $w^{\ep,\delta}$ is continuous or even measurable. Observe that the properties of $v_n^{\ep}$ are similar to those of an asymptotically subadditive cocycle, as defined and studied in \cite{FHu}.%
\end{remark}

Observe that Lemma \ref{lem_subadd} in the Appendix can be applied to $v := w^{\ep,\delta}:\N \tm \UC \rightarrow \R$, $(n,u) \mapsto w^{\ep,\delta}_n(u)$, and $f = \theta_1$. (To obtain a subadditive cocycle on $\Z_+ \tm X$, one might define $v_0(u) :\equiv 0$.) Indeed, by the preceding proposition, for all $n\in\N$, choosing $\delta = \delta(\alpha)$ sufficiently small, we have%
\begin{align}\label{eq_wub}
  \frac{1}{n}w^{\ep,\delta}_n(u) \leq \frac{1}{n}v_n^{\ep+\delta/2}(u) + \alpha + \zeta + \frac{1}{n}\log C_{\delta/2} \leq \frac{1}{n}(C_{\delta/2}+\vol(N_{\ep+\delta/2}(Q))) + \alpha + \zeta < \infty.%
\end{align}
Moreover, the definition of $w^{\ep,\delta}$ implies that for an appropriately chosen compact neighborhood $V$ of $Q$,%
\begin{align*}
  \frac{1}{n}w^{\ep,\delta}_n(u) &\geq \frac{1}{n}\inf\Bigl\{ \log|\alpha| + \log \min_{\varphi_{j,u}(x) \in V,\atop 0\leq j < n}J^+\varphi_{n,u}(x)^{-1} : \alpha \mbox{ \ldots } \Bigr\}\\
	&= \frac{1}{n}\log\min_{\varphi_{j,u}(x) \in V,\atop 0\leq j < n}J^+\varphi_{n,u}(x)^{-1} + \frac{1}{n}\inf\left\{ \log|\alpha| : \alpha \mbox{ \ldots } \right\}\\
	&\geq -\log \max_{(x,u') \in V \tm \UC}J^+\varphi_{1,u'}(x) > - \infty.%
\end{align*}

The next proposition shows that the $\limsup$ and the infimum in \eqref{eq_ielb} can indeed be interchanged if we additionally let $\ep$ become arbitrarily small.%

\begin{proposition}\label{prop_limsupliminf_lb}
The invariance entropy satisfies%
\begin{equation*}
  h_{\inv}(K,Q) \geq -\lim_{\ep\downarrow0}\sup_{u\in\UC}\liminf_{n\rightarrow\infty}\frac{1}{n}\log\vol\left(Q_{\rmd}(n,u,\ep)\right).%
\end{equation*}
\end{proposition}

\begin{proof}
As a consequence of Proposition \ref{prop_general_lb}, for sufficiently small $\ep>0$, we have%
\begin{equation}\label{eq_former_lb}
  h_{\inv}(K,Q) \geq -\liminf_{n\rightarrow\infty}\sup_{u\in\UC}\frac{1}{n}v_n^{\ep}(u).%
\end{equation}
We define%
\begin{equation*}
  S := \sup\Bigl\{\lambda\in\R\ :\ \exists u_k \in \UC,\ n_k \rightarrow \infty \mbox{ with } \lambda = \lim_{k\rightarrow\infty}\frac{1}{n_k}v^{\ep}_{n_k}(u_k) \Bigr\}.%
\end{equation*}
This number is finite and independent of $\ep$ (as long as $\ep$ is sufficently small). Finiteness follows from \eqref{eq_wub} and the inequality in Proposition \ref{prop_wprops}(ii). By Proposition \ref{prop_wprops}(iv), there exists $N = N(\delta)\in\N$ so that $v_n^{\ep}(u) \leq v_n^{\ep+\delta}(u) \leq C + v_{n-2N}^{\ep}(\theta_Nu)$ for all $u\in\UC$ and $n>2N$. This easily implies the independence of $S$ on $\ep$.%

Fix $\alpha>0$ and choose $\delta = \delta(\alpha)>0$ according to Proposition \ref{prop_wprops}(iii). Now consider a sequence $\rho_k \downarrow 0$ and sequences $u_k \in \UC$, $n_k \rightarrow \infty$ so that%
\begin{equation*}
  \frac{1}{n_k}w^{\ep,\delta}_{n_k}(u_k) > S - \zeta - \rho_k \mbox{\quad for all\ } k \geq 0,%
\end{equation*}
which is possible by Proposition \ref{prop_wprops}(ii). We put $\tilde{\rho}_k := 1/\sqrt{n_k}$. By Lemma \ref{lem_subadd}, we find times $n_k^*<n_k$ so that%
\begin{equation*}
  \frac{1}{l}w^{\ep,\delta}_l(\theta_{n_k^*}u_k) > S - \zeta - \rho_k - \tilde{\rho}_k \mbox{\quad for \ } 0 < l \leq n_k - n_k^*,%
\end{equation*}
where $n_k - n_k^* \geq \sqrt{n_k}/(2\omega)$ and $\omega$ is defined as in \eqref{eq_def_omega}. By Proposition \ref{prop_wprops}(iii), this implies%
\begin{equation*}
  \frac{1}{l}v^{\ep + \delta/2}_l(\theta_{n_k^*}u_k) > S - (2\zeta+\alpha) - \rho_k - \tilde{\rho}_k - \frac{1}{l}\log C_{\delta/2}.%
\end{equation*}
We put%
\begin{equation*}
  \tilde{u}_k := \theta_{n_k^*}u_k,\quad \tilde{n}_k := n_k - n_k^* \rightarrow \infty.%
\end{equation*}
By compactness, we may assume that $\tilde{u}_k$ converges to some $\tilde{u} \in \UC$. Fix $n\in\N$ and $\rho>0$. Then, for $k$ large enough, $n \leq \tilde{n}_k$ and, by continuity of $v^{\ep+\delta/2}_n$,%
\begin{equation*}
  \bigl|v_n^{\ep+\delta/2}(\tilde{u}) - v_n^{\ep+\delta/2}(\tilde{u}_k)\bigr| < \rho.%
\end{equation*}
We obtain%
\begin{align*}
  \frac{1}{n}v_n^{\ep+\delta/2}(\tilde{u}) &= \frac{1}{n}v_n^{\ep+\delta/2}(\tilde{u}_k) + \Bigl(\frac{1}{n}v_n^{\ep+\delta/2}(\tilde{u}) - \frac{1}{n}v^{\ep+\delta/2}_n(\tilde{u}_k)\Bigr)\\ &> S - (2\zeta+\alpha) - \rho_k - \tilde{\rho}_k - \frac{1}{n}\log C_{\delta/2} - \frac{\rho}{n}.%
\end{align*}
This implies%
\begin{equation*}
  \liminf_{n\rightarrow\infty}\frac{1}{n}v^{\ep+\delta/2}_n(\tilde{u}) \geq S - (2\zeta + \alpha).%
\end{equation*}
Now choose for each $n>0$ some $u_n^* \in \UC$ with $\sup_{u\in\UC}v_n^{\ep}(u)/n = v_n^{\ep}(u_n^*)/n$, which is possible, since $v^{\ep}_n$ is continuous. Then%
\begin{align*}
  \liminf_{n\rightarrow\infty}\sup_{u\in\UC}\frac{1}{n}v_n^{\ep}(u) &= \liminf_{n\rightarrow\infty}\frac{1}{n}v_n^{\ep}(u_n^*)\allowdisplaybreaks\\
	&\leq S \leq 2\zeta + \alpha + \liminf_{n\rightarrow\infty}\frac{1}{n}v^{\ep+\delta/2}_n(\tilde{u})\\ 
	&\leq 2\zeta + \alpha + \sup_{u\in\UC}\liminf_{n\rightarrow\infty}\frac{1}{n}v^{\ep+\delta/2}_n(u)\\
	&\leq 2\zeta + \alpha + \sup_{u\in\UC}\liminf_{n\rightarrow\infty}\frac{1}{n}\left(C + v^{\ep}_{n-2N}(\theta_Nu)\right)\\
	&= 2\zeta + \alpha + \sup_{u\in\UC}\liminf_{n\rightarrow\infty}\frac{1}{n}v^{\ep}_n(u).%
\end{align*}
Since $\alpha$ was chosen arbitrarily, this shows that%
\begin{equation*}
  \liminf_{n\rightarrow\infty}\sup_{u\in\UC}\frac{1}{n}v_n^{\ep}(u) \leq 2\zeta + \sup_{u\in\UC}\liminf_{n\rightarrow\infty}\frac{1}{n}v^{\ep}_n(u).%
\end{equation*}
Now, as $\ep$ tends to zero, also $\zeta$ can be chosen arbitrarily small. Combining this with \eqref{eq_former_lb}, the proof is complete.%
\end{proof}

\subsection{Relation to random escape rates and proof of the main result}\label{subsec_rer}

In the following, we use the notation $\MC_T$ for the set of all $T$-invariant Borel probability measures, where $T:X\rightarrow X$ is a continuous map on a compact metric space.%

\begin{proposition}\label{prop_supmeas}
The invariance entropy satisfies%
\begin{equation*}
  h_{\inv}(K,Q) \geq -\lim_{\ep\downarrow0}\sup_{P\in\MC_{\theta_1}}\liminf_{n\rightarrow\infty}\frac{1}{n}\int \log\vol\left(Q_{\rmd}(n,u,\ep)\right)\rmd P(u).%
\end{equation*}
\end{proposition}

\begin{proof}
Pick an arbitrary $u\in\UC$ and let%
\begin{equation*}
  \beta := \liminf_{n\rightarrow\infty}\frac{1}{n}v_n^{\ep}(u).%
\end{equation*}
Consider the sequence of Borel probability measures defined by%
\begin{equation*}
  P_n := \frac{1}{n}\sum_{k=0}^{n-1}\delta_{\theta_ku}.%
\end{equation*}
Since $\UC$ is compact, there exists a weak$^*$-limit point $P$ of $(P_n)_{n\in\Z_+}$. With standard arguments, one shows that $P$ is $\theta_1$-invariant. We claim that%
\begin{equation}\label{eq_claim}
  \liminf_{m\rightarrow\infty}\frac{1}{m}\int v_m^{\ep} \rmd P \geq \beta,%
\end{equation}
which obviously proves the assertion. To prove \eqref{eq_claim}, fix $\alpha>0$ and choose $\delta = \delta(\alpha)$ according to Proposition \ref{prop_wprops}(iii). Then consider the following chain of identities and inequalities for an arbitrary but fixed $m\geq1$:%
\begin{align*}
  \beta &= \liminf_{n\rightarrow\infty}\frac{1}{n}v_n^{\ep}(u)\allowdisplaybreaks\\
	&\leq \zeta + \liminf_{n\rightarrow\infty}\frac{1}{n}w_n^{\ep,\delta/2}(u)\allowdisplaybreaks\\
	&\leq \zeta + \liminf_{n\rightarrow\infty}\frac{1}{nm}\sum_{i=0}^{n-m}w_m^{\ep,\delta/2}(\theta_iu)\allowdisplaybreaks\\
	&= \zeta + \liminf_{n\rightarrow\infty}\frac{1}{nm}\sum_{i=0}^{n-1}w_m^{\ep,\delta/2}(\theta_iu)\allowdisplaybreaks\\
	&\leq 2\zeta + \alpha + \frac{\log C_{\delta/4}}{m} + \liminf_{n\rightarrow\infty}\frac{1}{nm}\sum_{i=0}^{n-1} v_m^{\ep+\delta/4}(\theta_iu)\allowdisplaybreaks\\
	&= 2\zeta + \alpha + \frac{\log C_{\delta/4}}{m} + \liminf_{n\rightarrow\infty}\frac{1}{m} \int v_m^{\ep+\delta/4} \rmd P_n.%
\end{align*}
The second line follows from Proposition \ref{prop_wprops}(ii) and the fifth from statement (iii) of the same proposition. The fourth line uses that $w_m^{\ep,\delta/2}$ is bounded on $\UC$ (see \eqref{eq_wub}) and the last line just uses the definition of $P_n$. It remains to show the third line. To this end, for each $i$ in the range $0 \leq i < m$ let us choose integers $q_i,r_i$ such that $n = i + q_im + r_i$ with $q_i \geq 0$ and $0 \leq r_i < m$. We have (see Lemma \ref{lem_combinatorics})%
\begin{equation*}
  \sum_{i=0}^{m-1}\sum_{j=0}^{q_i-1}w_m^{\ep,\delta/2}(\theta_{i+jm}u) = \sum_{i=0}^{n-m}w_m^{\ep,\delta/2}(\theta_iu).%
\end{equation*}
Hence, using subadditivity (Proposition \ref{prop_wprops}(i)), we find%
\begin{align*}
  m w_n^{\ep,\delta/2}(u) &\leq \sum_{i=0}^{m-1}\Bigl(w_i^{\ep,\delta/2}(u) + \sum_{j=0}^{q_i-1}w_m^{\ep,\delta/2}(\theta_{i+jm}u) + w_{r_i}^{\ep,\delta/2}(\theta_{i+q_im}u)\Bigr)\\
	          &= \sum_{i=0}^{m-1}w_i^{\ep,\delta/2}(u) + \sum_{i=0}^{n-m}w_m^{\ep,\delta/2}(\theta_iu) + \sum_{i=0}^{m-1}w^{\ep,\delta/2}_{r_i}(\theta_{n-r_i}u).%
\end{align*}
Dividing both sides by $mn$ and taking the $\liminf$ for $n\rightarrow\infty$ completes the proof of the third line above. We have shown that for every $\alpha>0$ there exists $\delta>0$ such that for all $m\geq1$,%
\begin{equation*}
  \liminf_{n\rightarrow\infty}\frac{1}{m} \int v_m^{\ep+\delta/4} \rmd P_n \geq \beta - (2\zeta + \alpha) -  \frac{\log C_{\delta/4}}{m},%
\end{equation*}
implying (by continuity of $v_m^{\ep+\delta/4}$)%
\begin{equation*}
  \frac{1}{m} \int v_m^{\ep+\delta/4} \rmd P \geq \beta - (2\zeta + \alpha) - \frac{\log C_{\delta/4}}{m}.%
\end{equation*}
According to Proposition \ref{prop_wprops}(iv) choose $N\in\N$ so that $v^{\ep+\delta/4}_m(u) \leq C + v^{\ep}_{m-2N}(\theta_Nu)$ for $m>2N$, which yields%
\begin{equation*}
  \frac{C}{m} + \frac{1}{m}\int v_{m-2N}^{\ep} \rmd P \geq \beta - (2\zeta + \alpha) - \frac{\log C_{\delta/4}}{m},%
\end{equation*}
where we use that $P$ is $\theta_1$-invariant. Letting $m\rightarrow\infty$ and subsequently $\alpha\downarrow0$, with Proposition \ref{prop_limsupliminf_lb} we arrive at%
\begin{equation*}
  h_{\inv}(K,Q) \geq -\lim_{\ep\downarrow0}\Bigl(2\zeta + \sup_{P\in\MC_{\theta_1}} \liminf_{n\rightarrow\infty}\frac{1}{n}\int v_n^{\ep}\rmd P\Bigr).%
\end{equation*}
Since $\zeta$ can be chosen arbitrarily small as $\ep$ tends to zero, the proof is complete.%
\end{proof}

\begin{remark}
The main ideas in the above proof are taken from \cite[Lem.~A.6]{Mor}.%
\end{remark}

Observe that the numbers%
\begin{equation*}
  \liminf_{n\rightarrow\infty}\frac{1}{n}\int \log\vol\left(Q_{\rmd}(n,u,\ep)\right) \rmd P(u)%
\end{equation*}
are random escape rates (cf., e.g., \cite{Liu}) and can be related to metric entropy and Lyapunov exponents by standard techniques, which we will now do.%

In the following, we regard the discretized control flow $\Phi_n:\UC \tm M \rightarrow \UC \tm M$, $n\in\Z$, as a smooth random dynamical system over the base $(\UC,\BC(\UC),P,\theta_1)$, where $P$ is a $\theta_1$-invariant Borel probability measure on $\UC$, as explained in Section \ref{sec_framework}. We will use the convention to denote $\theta_1$-invariant measures on $\UC$ by $P$ and $\Phi_1$-invariant measures on $\UC \tm M$ by $\mu$.%

We can finally prove our main result (announced as Theorem \ref{thm_mainresult} in Section \ref{sec_framework}):%

\begin{theorem}
The invariance entropy satisfies%
\begin{equation*}
  h_{\inv}(K,Q) \geq \inf_{\mu}\left(\int \log J^+\varphi_{1,u}(x)\rmd\mu(u,x) - h_{\mu}(\varphi)\right),%
\end{equation*}
where the infimum is taken over all $\Phi_1$-invariant measures $\mu$ with $\supp\mu \subset L(Q)$.%
\end{theorem}

\begin{proof}
Fix some $P\in\MC_{\theta_1}$ and a sufficiently small $\ep>0$. For every $n\in\N$ and $u\in\UC$ let $E_n(u) \subset Q_{\rmd}(n,u,\ep)$ be a maximal $(n,u,\delta)$-separated set for some sufficiently small $\delta>0$. Then $E_n(u)$ also $(n,u,\ep)$-spans $Q_{\rmd}(n,u,\ep)$ and hence, by the volume lemma,%
\begin{equation}\label{eq_vnep_ineq}
  v_n^{\ep}(u) \leq \log \sum_{x \in E_n(u)}\vol(B^{n,u}_{\delta}(x)) \leq \log C_{\delta} + \log\sum_{x \in E_n(u)}\rme^{n\zeta -\log J^+\varphi_{n,u}(x)}.%
\end{equation}
We define sequences of probability measures on $M$ by%
\begin{equation*}
  \eta_n^u := \frac{\sum_{x \in E_n(u)} \rme^{n\zeta-\log J^+\varphi_{n,u}(x)} \delta_x}{\sum_{x\in E_n(u)} \rme^{n\zeta-\log J^+\varphi_{n,u}(x)}},\quad n\in\N,\ u\in\UC,%
\end{equation*}
and%
\begin{equation*}
  \nu_n^u := \frac{1}{n}\sum_{i=0}^{n-1}\varphi(-i,\cdot,u)^{-1}_*\eta_n^{\theta_{-i}u}.%
\end{equation*}
To obtain an invariant measure for the discretized control flow on $\UC \tm M$ with marginal $P$ on $\UC$, we apply Lemma \ref{lem_crauel} (see Appendix) with $\Phi_1$ in place of $\Theta$ and random measures $\sigma_n$ given by $\rmd\sigma_n = \rmd \eta_n^u \rmd P(u)$. With the same arguments as used in the proof of the variational principle for pressure of random dynamical systems (cf.~\cite{Bog,ZCa}) one can choose the sets $E_n(u)$ such that $\eta_n^u$ depends measurably on $u$, i.e., such that $\sigma_n$ is a random measure for each $n$. Then%
\begin{equation*}
  \frac{1}{n}\sum_{i=0}^{n-1}(\Phi_i)_*\sigma_n = \frac{1}{n}\sum_{i=0}^{n-1}\sigma_n \circ \Phi_i^{-1},%
\end{equation*}
and hence (using that $\varphi(i,\cdot,\theta_{-i}v) = \varphi(-i,\cdot,v)^{-1}$),%
\begin{align*}
  \frac{1}{n}\sum_{i=0}^{n-1}(\Phi_i)_*\sigma_n(A) &= \frac{1}{n}\sum_{i=0}^{n-1} \sigma_n(\Phi_i^{-1}(A))\allowdisplaybreaks\\
	&= \frac{1}{n}\sum_{i=0}^{n-1} \int_{\UC} \int_M \unit_{\Phi_i^{-1}(A)}(u,x) \rmd \eta_n^u(x) \rmd P(u)\allowdisplaybreaks\\
	&= \frac{1}{n}\sum_{i=0}^{n-1} \int_{\UC} \int_M \unit_A(\theta_i u,\varphi(i,x,u)) \rmd \eta_n^u(x) \rmd P(u)\allowdisplaybreaks\\
	&= \frac{1}{n}\sum_{i=0}^{n-1} \int_{\UC} \int_M \unit_A(v,\varphi(i,x,\theta_{-i}v)) \rmd \eta_n^{\theta_{-i}v}(x) \rmd P(v)\allowdisplaybreaks\\
	&= \frac{1}{n}\sum_{i=0}^{n-1} \int_{\UC} \int_M \unit_A(v,y) \rmd\left[ \varphi(i,\cdot,\theta_{-i}v)_* \eta_n^{\theta_{-i}v} \right](y) \rmd P(v)\allowdisplaybreaks\\
	&= \int_{\UC} \int_M \unit_A(v,y) \rmd \nu_n^v(y) \rmd P(v).%
\end{align*}
Consequently, the measures $\nu_n^u$, $u\in\UC$, are the sample measures of $(1/n)\sum_{i=0}^{n-1}(\Phi_i)_*\sigma_n$, and Lemma \ref{lem_crauel} shows that any limit point of this sequence is a $\Phi_1$-invariant measure with marginal $P$. We fix such a limit point $\mu$ that is obtained from a subsequence $(n_j)$.%

We can choose a finite measurable partition $\PC = \{P_1,\ldots,P_k\}$ of $M$ with $\diam\PC < \delta$ and $\mathrm{Pr}_M \mu(\partial P_i) = 0$ for $1 \leq i \leq k$ (see, e.g., \cite{KHa}). Since $\mathrm{Pr}_M\mu(\partial P_i) = \int \mu_u(\partial P_i)\rmd P(u)$, we have%
\begin{equation*}
  \mu_u(\partial P_i) = 0 \quad P\mbox{-a.s.}%
\end{equation*}
Put $\alpha_n(u,x) := n\zeta - \log J^+\varphi_{n,u}(x)$ and $S:=\sum_{x\in E_n(u)} \rme^{\alpha_n(u,x)}$. Since each element of $\bigvee_{i=0}^{n-1}\varphi(i,\cdot,u)^{-1}\PC$ contains at most one element of $E_n(u)$, we obtain for $P$-almost every $u\in\UC$ that%
\begin{equation*}
  H_{\eta_n^u}\left(\bigvee_{i=0}^{n-1}\varphi(i,\cdot,u)^{-1}\PC\right) - \int (-\alpha_n(u,x))\rmd\eta_n^u(x) = -\sum_{x\in E_n(u)} \frac{\rme^{\alpha_n(u,x)}}{S} \log \frac{\rme^{\alpha_n(u,x)}}{S} = \log S.%
\end{equation*}
Using precisely the same arguments as in \cite[Thm.~6.1]{Bog}, with \eqref{eq_vnep_ineq} we can conclude that%
\begin{equation*}
  \frac{1}{n} \int \log\sum_{x\in E_n(u)}\rme^{\alpha_n(u,x)} \rmd P(u) \leq \frac{1}{n}\int H_{\mu_u}\left(\bigvee_{i=0}^{n-1}\varphi(i,\cdot,u)^{-1}\PC\right)\rmd P(u) + \int \alpha_1 \rmd \mu
\end{equation*}
for every $n\in\N$. Taking the $\liminf$ for $n\rightarrow\infty$ and using that $h_{\mu}(\varphi;\PC) \leq h_{\mu}(\varphi)$, we arrive at%
\begin{equation*}
  \liminf_{n\rightarrow\infty}\frac{1}{n}\int v_n^{\ep}(u) \rmd P(u) \leq h_{\mu}(\varphi) - \int \log J^+\varphi_{1,u}(x)\rmd\mu(u,x) + \zeta.%
\end{equation*}
Since $\zeta$ can be chosen arbitrarily small as $\ep$ tends to zero, we end up with the desired estimate for the invariance entropy. It only remains to show that $\supp \mu \subset L(Q)$. To show this, observe that%
\begin{equation*}
  \supp\mu_u \subset N_{\ep}(Q) \mbox{\quad for all\ } u \in \UC.%
\end{equation*}
Furthermore, we have the identities $\varphi(n,\cdot,u)_*\mu_u = \mu_{\theta_nu}$ for all $n\in\Z$ $P$-almost surely (cf.~\cite[Thm.~1.4.5]{Arn}), implying $\varphi(n,\supp\mu_u,u) = \supp\mu_{\theta_nu}$. Since $Q$ is isolated, this yields $\supp\mu_u \subset Q$ if $\ep$ is chosen sufficiently small, and hence $\supp\mu_u \subset Q(u)$. Consequently, $\mu(L(Q)) = \int \mu_u(Q(u))\rmd P(u) = 1$, implying $\supp\mu \subset L(Q)$.%
\end{proof}

\section{Examples}\label{sec_example}

Examples for the application of our result can be found among invariant systems on flag manifolds of noncompact semisimple Lie groups (i.e., systems induced by control-affine systems on the Lie group whose drift and control vector fields are all either right- or left-invariant). Since the treatment of such systems necessitates sophisticated tools from Lie theory, we do not include it in this paper. Instead, we only consider a special subclass of such systems that can be described without using Lie theory. Starting with a bilinear system%
\begin{equation}\label{eq_bcs}
  \dot{x}(t) = \Bigl(A_0 + \sum_{i=1}^m u_i(t)A_i\Bigr)x(t),\quad u\in\UC,%
\end{equation}
on $\R^{d+1}$, we can study the system induced on the $d$-dimensional real projective space $\R\P^d$. In the case when $\tr A_i = 0$ for $i=0,1,\ldots,m$, this corresponds to a right-invariant system on $\SL(d+1,\R)$, however, we can in fact do without this assumption.%

The control flow $\Phi_t:\UC \tm \R^{d+1} \rightarrow \UC \tm \R^{d+1}$, $t\in\R$, of \eqref{eq_bcs} can be regarded as a linear flow on the trivial vector bundle $\UC \tm \R^{d+1}$, hence it induces a flow on the projective bundle $\UC \tm \R\P^d$. It is easy to see that this induced flow, which we denote by $\Phi^{\P} = (\theta,\varphi^{\P})$, is the control flow of the control-affine system on $\R\P^d$ with vector fields $f_0,f_1,\ldots,f_m$, given by $f_i(\P x) = \rmD\P(x) A_i x$, where $\P:\R^{d+1}\backslash\{0\}\rightarrow \R\P^d$ denotes the canonical projection.%

Since the shift flow $\theta$ acting on the base $\UC$ is chain transitive (cf.~\cite[Prop.~4.1.1]{CKl}), we can apply Selgrade's theorem (cf.~\cite[Thm.~5.2.5]{CKl}) and obtain a decomposition%
\begin{equation*}
  \UC \tm \R^{d+1} = \WC^1 \oplus \cdots \oplus \WC^r%
\end{equation*}
into exponentially separated invariant subbundles, each of which corresponds to a chain recurrent component on $\UC \tm \R\P^d$, and hence a chain control set on $\R\P^d$. (The chain control sets are the projections of the chain recurrent components of the control flow to the state space, cf.~\cite{CKl}.) The chain recurrent component associated with $\WC^i$ is given by%
\begin{equation}\label{eq_prmorsecomp}
  L(Q_i) = \left\{ (u,\P x) \in \UC \tm \R\P^d\ :\ x \neq 0 \mbox{ and } (u,x) \in \WC^i \right\}.%
\end{equation}
The chain control sets $Q_1,\ldots,Q_r$ are the canonical projections of $L(Q_1),\ldots,L(Q_r)$ to $\R\P^d$. They satisfy assumption (P1) of our main result by definition. Now we fix $i \in \{1,\ldots,r\}$ and put%
\begin{equation*}
  L(Q) := L(Q_i),\quad \VC^0 := \WC^i,\quad \VC^+ := \WC^1 \oplus \cdots \oplus \WC^{i-1},\quad \VC^- := \WC^{i+1} \oplus \cdots \oplus \WC^r,%
\end{equation*}
where we assume that the exponentially separated subbundles $\WC^i$ are ordered by increasing growth rates. Then, for each $(u,\P x) \in L(Q)$, we put%
\begin{equation*}
  E^0(u,\P x) := \rmD \P(x)\WC^0(u),\quad E^{\pm}(u,\P x) := \rmD \P(x)\WC^{\pm}(u).%
\end{equation*}
By \cite[Prop.~7.8 and Prop.~7.9]{Kaw}, we have the following result.%

\begin{proposition}
The following assertions hold:%
\begin{enumerate}
\item[(i)] $E^0(u,\P x)$ and $E^{\pm}(u,\P x)$ are well-defined linear subspaces of $T_{\P x}\R\P^d$. Their dimensions are constant on $L(Q)$ with%
\begin{equation*}
  \dim E^0(u,\P x) = \rk \VC^0 - 1,\quad \dim E^{\pm}(u,\P x) = \rk \VC^{\pm}.%
\end{equation*}
\item[(ii)] For each $(u,\P x) \in L(Q)$, we have a direct sum decomposition%
\begin{equation*}
  T_{\P x}\R\P^d = E^-(u,\P x) \oplus E^0(u,\P x) \oplus E^+(u,\P x).%
\end{equation*}
\item[(iii)] The spaces $E^0(u,\P x)$ and $E^{\pm}(u,\P x)$ are the fibers of subbundles $E^0 \rightarrow L(Q)$ and $E^{\pm} \rightarrow L(Q)$, and satisfy%
\begin{equation*}
  \rmD\varphi^{\P}_{t,u}(\P x)E^i(u,\P x) = E^i(\Phi_t^{\P}(u,\P x)),\quad t \in \R,\ i \in \{-,0,+\}.%
\end{equation*}
\item[(iv)] The restriction of the derivative $\rmD\varphi^{\P}_{(\cdot,\cdot)}$ to $E^-$ is uniformly contracting and the corresponding restriction to $E^+$ is uniformly expanding. In particular, if $\rk\VC^0 = 1$, then $Q$ is uniformly hyperbolic without center bundle.%
\end{enumerate}
\end{proposition}

Hence, we can easily describe a class of chain control sets on $\R\P^d$ which satisfy the partial hyperbolicity assumption (P2) of our main result with $E^{0-} := E^0 \oplus E^-$. These are precisely the chain control sets with subexponential growth rates on $E^0$. (A Lie-algebraic characterization in terms of the so-called flag type of the control flow will be presented in \cite{DK4}.)%

Now we show that each chain control $Q_i$ also satisfies the assumptions (P3) and (P4).%

\begin{proposition}\label{prop_ccs_properties}
The following statements hold for the chain control set $Q = Q_i$:%
\begin{enumerate}
\item[(i)] The set-valued mapping $u \mapsto Q(u)$ is lower semicontinuous.%
\item[(ii)] $Q$ is isolated, i.e., there exists a neighborhood $N \subset \R\P^d$ of $Q$ so that $\varphi^{\P}(\R,\P x,u) \subset N$ implies $\varphi^{\P}(\R,\P x,u) \subset Q$ for any $(u,\P x)\in\UC\tm\R\P^d$.%
\end{enumerate}
\end{proposition}

\begin{proof}
(i) First notice that the sets $Q(u)$ are all nonempty. In fact, we can write%
\begin{equation}\label{eq_qiudescr}
  Q(u) = \P\WC^i(u) = \left\{ \P x \in \R\P^d\ :\ x \in \WC^i(u) \backslash \{0\} \right\}.%
\end{equation}
Indeed, if $\varphi^{\P}(\R,\P x,u) \subset Q$, then $(u,\P x) \in L(Q)$ and \eqref{eq_prmorsecomp} implies $(u,x) \in \WC^i$, which is equivalent to $x \in \WC^i(u)$. Conversely, if $0 \neq x \in \WC^i(u)$, then $\varphi_{t,u}(x) \in \WC^i(\theta_tu)$ for all $t\in\R$ by invariance of $\WC^i$. This implies $\varphi^{\P}_{t,u}(\P x) = \P\varphi_{t,u}(x) \in \P\WC^i(\theta_tu) \subset Q_i$, using \eqref{eq_prmorsecomp} again. Now it suffices to prove that $u\mapsto Q(u)$ is inner semicontinuous, since this mapping is compact-valued with values in a compact metric space. Pick $u\in\UC$ and $\P x \in Q_i(u)$. We have to show that for any sequence $u_k \rightarrow u$ in $\UC$ there are $\P x_k \in Q_i(u_k)$ with $\P x_k \rightarrow \P x$. First note that $x \in \WC^i(u)$ by \eqref{eq_qiudescr}. Since $\WC^i$ is a vector bundle, we find $x_k \in \WC^i(u_k)$ with $x_k \rightarrow x$. Hence, $\P x_k \rightarrow \P x$ and $\P x_k \in Q_i(u_k)$.%

(ii) We use that the sets $L(Q_1),\ldots,L(Q_r)$ form a Morse decomposition for the control flow on $\UC \tm \R\P^d$ (since they are the chain recurrent components and $r<\infty$). We have $Q = Q_i = \pi_{\R\P^d}(L(Q_i))$. Now consider some $(u,\P x) \in \UC \tm \R\P^d$, not contained in a Morse set. Then the $\alpha$- and $\omega$-limit sets $\alpha(u,\P x)$ and $\omega(u,\P x)$ are contained in some $L(Q_{i_1})$ and $L(Q_{i_2})$ with $i_1 \neq i_2$, respectively. Hence, their projections to $\R\P^d$ are contained in the corresponding (disjoint) chain control sets $Q_{i_1}$ and $Q_{i_2}$. Consequently, if $\varphi^{\P}(\R,\P x,u)$ is contained in a neighborhood of $Q_i$ whose closure intersects no other chain control set, then $\varphi^{\P}(\R,\P x,u) \subset Q_i$.%
\end{proof}

Actually, for the systems on $\R\P^d$ our lower bound simplifies, since the entropy term $h_{\mu}(\varphi^{\P})$ always vanishes.%

\begin{theorem}
If $Q = Q_i$ is a chain control set on $\R\P^d$, which satisfies assumption (P2), then for any $K \subset Q$ of positive volume,%
\begin{equation}\label{eq_psp_lb}
  h_{\inv}(K,Q) \geq \inf_{\mu} \int \log J^+\varphi^{\P}_{1,u}(x) \rmd\mu(u,x) = \inf_{(u,x)\in L(Q)}\limsup_{\tau\rightarrow\infty}\frac{1}{\tau}\log J^+\varphi^{\P}_{\tau,u}(x).%
\end{equation}
\end{theorem}

\begin{proof}
To show the inequality in \eqref{eq_psp_lb}, it suffices to prove that $h_{\mu}(\varphi^{\P}) = 0$ for all $\mu$. To this end, observe that $h_{\mu}(\varphi^{\P})$ is bounded from above by the topological entropy $h_{\tp}(\varphi^{\P})$ of the corresponding bundle RDS on $L(Q)$, following from the variational principle for bundle RDS (cf.~\cite[Thm.~1.2.13]{HKa}). We show that $h_{\tp}(\varphi^{\P}) = 0$ for each fixed invariant measure $P$ on the base space $\UC$. To this end, equip $\R\P^d$ with the standard metric induced by the round metric on the unit sphere in $\R^{d+1}$. Consider two points $\P x,\P y$ on the same fiber $Q(u)$, $u\in\UC$. The fibers are projections of linear subspaces by \eqref{eq_qiudescr}, hence they are totally geodesic submanifolds. Now take a shortest geodesic $\gamma:[0,1] \rightarrow Q(u)$ from $\P x$ to $\P y$ and let $n\in\N$. Then%
\begin{align*}
  d(\varphi_{n,u}^{\P}(\P x),\varphi_{n,u}^{\P}(\P y)) &\leq \length(\varphi_{n,u}^{\P} \circ \gamma) = \int_0^1 |\rmD\varphi_{n,u}^{\P}(\gamma(t))\dot{\gamma}(t)| \rmd t \\
	 &\leq \sup_{\P z \in Q(u),\ w\in T_{\P z}Q(u),\ |w| = 1} |\rmD\varphi_{n,u}^{\P}(\P z)w| \cdot \length(\gamma).%
\end{align*}
Now observe that $\length(\gamma) = d(\P x,\P y)$ and $T_{\P z}Q(u) = E^0(u,\P z)$. By assumption (P2) we can thus choose $n$ large enough (independently of $u,\P z$ and $w$) so that $|\rmD\varphi_{n,u}^{\P}(\P z)w| \leq \rme^{\ep n}|w|$. Hence,%
\begin{equation*}
  d(\varphi_{n,u}^{\P}(\P x),\varphi_{n,u}^{\P}(\P y)) \leq \rme^{\ep n} d(\P x,\P y).%
\end{equation*}
By standard methods, one shows that this implies $h_{\tp}(\varphi^{\P}) \leq \dim Q(u) \cdot \ep$, and since $\ep>0$ was chosen arbitrarily, $h_{\tp}(\varphi^{\P}) = 0$ follows. The equality in \eqref{eq_psp_lb} follows from the general theory of continuous additive cocycles, see, e.g., \cite{SM3}.%
\end{proof}

According to \cite{DK1,DK2}, the lower bound in the above theorem coincides with the actual value of $h_{\inv}(K,Q)$ in the case when $Q$ is a uniformly hyperbolic (without center bundle) chain control set and the Lie algebra rank condition holds on $\inner Q$. In the more general case, when (P2) is satisfied, it is very likely that we also have equality in \eqref{eq_psp_lb}, since the results of \cite{DK1} (in the case when $Q$ is the closure of a control set) already imply%
\begin{equation*}
  h_{\inv}(K,Q) \leq \inf_{(u,x)}\limsup_{\tau\rightarrow\infty}\frac{1}{\tau}\log J^+\varphi^{\P}_{\tau,u}(x),%
\end{equation*}
the infimum taken over all $(u,x) \in L(Q)$ so that $u \in \inner_{L^{\infty}}\UC$ and $\varphi(\R_+,x,u)$ contained in a compact subset of $\inner Q$, cf.~\cite[Thm.~3.5]{DK1}.%

\section{Remarks, interpretation and further directions}\label{sec_interpret}%

\paragraph{Relation to submanifold stabilization.} A problem, different from set-invariance, that seems to be closely connected to the analysis in this paper is the stabilization of a control system to an embedded submanifold of the state space. Traditionally, in control theory only the stabilization to equilibrium points or periodic solutions is studied. However, there are important applications, including synchronization, path following and pattern generation, where the desired control objective is actually the stabilization to a submanifold invariant under some constant control input. The works \cite{EMa,Ma1,Ma2,MAl} studied feedback stabilization to submanifolds. So far, it seems that stabilization to submanifolds has not been studied in a networked control setup, where the controller only has quantized state information available. This first necessitates a precise definition of the stabilizability to a submanifold via open-loop controls, which can be formulated as follows.%

\begin{definition}
A control-affine system $\Sigma$ is \emph{open-loop stabilizable} to an embedded submanifold $S$ of its state space $M$ if for every neighborhood $N \subset M$ of $S$ there is a neighborhood $N' \subset N$ so that the following holds: For every $x\in N'$ there is $u\in\UC$ with $\varphi(\R_+,x,u) \subset N$ and $\dist(\varphi(t,x,u),S) \rightarrow 0$ as $t\rightarrow\infty$.%
\end{definition}

Observing that the fibers $Q(u)$ of the chain control set can be submanifolds (which is the case in the examples of the preceding section), the proof of our main result suggests that the following conjecture holds.%

\begin{conjecture}
Consider a control-affine system $\Sigma$ with state space $M$ and assume the following:%
\begin{enumerate}
\item[(i)] There exists a compact embedded submanifold $S \subset M$, invariant under a constant control function with value $u_*\in\inner U$. (We let $(\phi_t)_{t\in\R}$ denote the flow associated with $u_*$.)%
\item[(ii)] $\Sigma$ is open-loop stabilizable to $S$.%
\item[(iii)] There exists a $\rmD\phi_t$-invariant decomposition%
\begin{equation*}
  T_xM = E^{0-}(x) \oplus E^+(x),\quad \forall x\in S.%
\end{equation*}
\item[(iv)] There exists a constant $\lambda>0$ such that for every $\ep>0$ there is $T>0$ with the following property. For all $x\in S$ and $t\geq T$,%
\begin{align*}
  |\rmD\phi_t(x)v| \geq \rme^{\lambda t}|v| &\mbox{\quad if\ } v \in E^+(x),\\
  |\rmD\phi_t(x)v| \leq \rme^{\ep t}|v| &\mbox{\quad if\ } v \in E^{0-}(x).%
\end{align*}
\end{enumerate}
Then the smallest data rate above which stabilization to $S$ can be achieved by a controller receiving state information over a noiseless digital channel, satisfies%
\begin{equation*}
  R_0 \geq -P_{\tp}((\phi_1)_{|S};-\log J^+\phi_1(x)).%
\end{equation*}
\end{conjecture}

\paragraph{The assumption of lower semicontinuity.} The assumption (P3) that the fiber map $u \mapsto Q(u)$ is lower semicontinuous seems to be quite strong. In general, we do not see any reason why a partially hyperbolic controlled invariant set or a chain control set should satisfy such an assumption. Even in the uniformly hyperbolic case (with one-dimensional center bundle), shadowing techniques can only show that $u \mapsto Q(u)$ is lower semicontinuous w.r.t.~the $L^{\infty}$-topology on $\UC$. The assumption of lower semicontinuity is needed though in the proof of Proposition \ref{prop_general_lb}, and to us it seems that there is no way around it. Fortunately, the chain control sets of the systems on $\R\P^d$ induced by bilinear systems, and more general, of invariant systems on flag manifolds of semisimple Lie groups, satisfy this assumption and we are still looking for a deeper reason or hidden mechanism which causes the lower semicontinuity for these systems.%

\paragraph{Topological pressure and SRB measures.} The lower bound obtained for $h_{\inv}(K,Q)$ is related to the topological pressure of random dynamical systems associated to the control flow. It would, of course, be desirable to gain a better understanding of this bound, which suggests that not only the instability of the dynamics on the set $Q$ (measured by $\int \log J^+\varphi_{1,u}\rmd \mu$) is relevant for the value of $h_{\inv}(Q)$. In fact, it is clear that in the extremal case, when $Q$ coincides with the state space $M$, the invariance entropy is zero, no matter how unstable the dynamics on $M$ is.%
 
For a fixed invariant measure $P$ on the base space $\UC$, and a fixed measure $\mu$ on $\UC \tm M$, invariant under $\Phi_1$, and projecting to $P$, observe that%
\begin{equation*}
  h_{\mu}(\varphi) - \int_{L(Q)} \log J^+\varphi_{1,u}(x) \rmd\mu(u,x) = 0%
\end{equation*}
if and only if $\mu$ is an SRB measure for the random dynamical system arising by equipping $\UC$ with $P$ and discretizing in time, cf.~\cite[Thm.~3.2.4]{HKa}. Since, for any measure $\mu$ the inequality%
\begin{equation*}
  h_{\mu}(\varphi) - \int_{L(Q)} \log J^+\varphi_{1,u}(x) \rmd\mu(u,x) \leq 0%
\end{equation*}
holds (Ruelle's inequality), the statement%
\begin{equation*}
  \sup_{\mu:\ (\mathrm{pr}_{\UC})_*=P \atop \supp \mu \subset L(Q)}\Bigl( h_{\mu}(\varphi) - \int_{L(Q)} \log J^+\varphi_{1,u}(x) \rmd\mu(u,x) \Bigr) = 0%
\end{equation*}
is equivalent to the existence of an SRB measure $\mu$ for the corresponding random dynamical system on the invariant set $L(Q)$ (in case that the unstable bundle does not vanish). Hence, the lower bound obtained for $h_{\inv}(K,Q)$ is zero iff for every invariant measure $P$ on $\UC$ there exists an associated SRB measure supported on $L(Q)$.%

\paragraph{Asymptotically subadditive cocycles.} The proofs of Propositions \ref{prop_limsupliminf_lb} and \ref{prop_supmeas} partially generalize known results for asymptotically subadditive cocycles (ASC), cf.~\cite{FHu,CZh}. An ASC over a continuous dynamical system $T:X\rightarrow X$ is a continuous map $\phi:\Z_+ \tm X \rightarrow \R$ so that for every $\ep>0$ there exists a continuous subadditive cocycle $\psi:\Z_+ \tm X \rightarrow \R$ over $T$ with%
\begin{equation*}
  \limsup_{n\rightarrow\infty}\frac{1}{n}\sup_{x\in X} |\phi_n(x) - \psi_n(x)| \leq \ep.%
\end{equation*}
Looking at Proposition \ref{prop_wprops}, we see that the family $(v_n^{\ep})_{\ep>0}$ has properties close to an ASC. Even though the approximating subadditive cocycles $w^{\ep,\delta}_n$ are not known to be continuous or even measurable, and the approximation in the sense above is not exactly satisfied (only in the limit as $\ep,\delta\downarrow0$), we can still prove some of the properties known for ASC such as the interchangeability of $\liminf$ and $\sup$ (Proposition \ref{prop_limsupliminf_lb}) and the expression of the supremal growth rate as the supremal ergodic average (Proposition \ref{prop_supmeas}).%

\section{Appendix}\label{sec_appendix}%

\subsection{Miscellaneous}%

The following lemma seems to be known, but we could not find a reference in the literature.%

\begin{lemma}\label{lem_volzero}
Let $M$ be a Riemannian manifold and $K \subset M$ a nonempty compact subset. Then for every $\ep>0$, the boundary of $N_{\ep}(K)$ has volume zero.%
\end{lemma}

\begin{proof}
We give the proof for $M = \R^n$ with the Euclidean metric. The general case can be proved by replacing straight lines with geodesics. Hence, let $K \subset \R^n$ be a nonempty compact set and $\ep>0$. Take $x \in \partial N_{\ep}(K)$ and fix a point $y\in K$ such that $\dist(x,K) = \|x-y\| = \ep$. We claim that the open ball $B_{\ep}(y)$ is contained in $N_{\ep}(K)$ and does not contain any point from $\partial N_{\ep}(K)$. Indeed, if $z \in B_{\ep}(y)$, then $\dist(z,K) \leq \|z-y\| < \ep$ and all points $w \in \partial N_{\ep}(K)$ satisfy $\dist(w,K) = \ep$ implying $\|w-y\|\geq\ep$. Let $r \in (0,\ep)$. Then the intersection $B_r(x) \cap B_{\ep}(y)$ contains the ball $B_{r/2}(t x + (1-t)y)$ with $t := 1 -r/(2\ep)$. Indeed, if $w \in B_{r/2}(tx+(1-t)y)$, then%
\begin{align*}
  \|w - x\| &\leq \|w - tx - (1-t)y\| + \|tx + (1-t)y - x\|\\
	&< \frac{r}{2} + (1-t) \|x - y\| =\frac{r}{2} +  \frac{r}{2\ep}\ep = r,\\
  \|w - y\| &\leq \|w - tx - (1-t)y\| + \|tx + (1-t)y - y\|\\
	&< \frac{r}{2} + t\|x - y\| = \frac{r}{2} + \left(\ep - \frac{r}{2}\right) = \ep.%
\end{align*}
Hence, for all $r\in(0,\ep)$ we have%
\begin{equation*}
  \frac{\vol(B_r(x) \cap \partial N_{\ep}(K))}{\vol(B_r(x))} \leq \frac{ c r^n - c (r/2)^n }{ c r^n } = 1 - 2^{-n} < 1.%
\end{equation*}
This proves that the density $d(x) = \lim_{r\downarrow0}\vol(B_r(x)\cap\partial N_{\ep}(K))/\vol(B_r(x))$ is less than one wherever it exists on $\partial N_{\ep}(K)$. Lebesgue's density theorem asserts that $d(x) = 1$ at almost every point of $\partial N_{\ep}(K)$. This can only be the case if $\vol(\partial N_{\ep}(K)) = 0$.%
\end{proof}

The next lemma is essentially taken from \cite[Lem.~2.4]{Gru}.%

\begin{lemma}\label{lem_subadd}
Let $f:X \rightarrow X$ be a map on some set $X$ and $v:\Z_+ \tm X \rightarrow \R$ a subadditive cocycle over $f$, i.e.,%
\begin{equation*}
  v_{n+m}(x) \leq v_n(x) + v_m(f^n(x)) \mbox{\quad for all\ } x\in X,\ n,m\in\Z_+.%
\end{equation*}
Additionally suppose that%
\begin{equation}\label{eq_def_omega}
  \omega := \sup_{(n,x) \in \N\tm X}\frac{1}{n}|v_n(x)| < \infty.%
\end{equation}
Then for every $x\in X$, $n\in\N$ and $\ep \in (0,2\omega)$ there is a time $0\leq n_1 < n$ with%
\begin{equation*}
  \frac{1}{k}v_k(f^{n_1}(x)) > \frac{1}{n}v_n(x) - \ep \mbox{\ \ for all\ } 0 < k \leq n-n_1.%
\end{equation*}
Moreover, $n - n_1 \geq (\ep n)/(2\omega) \rightarrow \infty$ for $n \rightarrow \infty$.%
\end{lemma}

\begin{proof}
We write $\sigma := v_n(x)/n$ and define%
\begin{equation*}
  \gamma := \min_{0 < k \leq n}\frac{1}{k}v_k(x).%
\end{equation*}
If $\gamma \geq \sigma - \ep$, the assertion follows with $n_1 = 0$. For $\gamma < \sigma - \ep$, observing that the minimum cannot be attained at $k = n$, let%
\begin{equation*}
  n_1 := \max\Bigl\{ k \in (0,n) \cap \Z \ :\ \frac{1}{k}v_k(x) \leq \sigma - \ep \Bigr\},%
\end{equation*}
implying $v_{n_1}(x)/n_1 \leq \sigma - \ep$. We obtain%
\begin{align*}
  \ep &\leq \frac{1}{n}v_n(x) - \frac{1}{n_1}v_{n_1}(x) = \frac{1}{n}v_{n_1 + (n - n_1)}(x) - \frac{1}{n_1}v_{n_1}(x)\\
			&\leq \frac{1}{n}\left(v_{n_1}(x) + v_{n-n_1}(f^{n_1}(x))\right) - \frac{1}{n_1}v_{n_1}(x)\\
			   &= \frac{1}{n}\left(-\frac{n-n_1}{n_1}v_{n_1}(x) + \frac{n-n_1}{n-n_1}v_{n-n_1}(f^{n_1}(x))\right)\\
				 &= \frac{n-n_1}{n}\left(\frac{1}{n-n_1}v_{n-n_1}(f^{n_1}(x)) - \frac{1}{n_1}v_{n_1}(x)\right) \leq 2\omega \frac{n-n_1}{n}.%
\end{align*}
This implies%
\begin{equation*}
  n - n_1 \geq \frac{\ep n}{2\omega} \rightarrow \infty \mbox{\quad for\ } n \rightarrow \infty.%
\end{equation*}
For $0 < k \leq n-n_1$ we have $v_{k+n_1}(x)/(k+n_1) > \sigma - \ep$ and this yields%
\begin{align*}
  \frac{1}{k}v_k(f^{n_1}(x)) &\geq \frac{1}{k}\left(v_{k+n_1}(x) - v_{n_1}(x)\right)\\
	                              &> \frac{1}{k}\left((k+n_1)(\sigma - \ep) - n_1(\sigma - \ep)\right) = \sigma - \ep,%
\end{align*}
completing the proof.%
\end{proof}

The following lemma is used in \cite[Lem.~A.6]{Mor}.%

\begin{lemma}\label{lem_combinatorics}
Let $n>m$ be positive integers. For each $i$ in the range $0 \leq i < m$ choose integers $q_i,r_i$ such that $n = i + q_im + r_i$ with $q_i \geq 0$ and $0 \leq r_i < m$. Then%
\begin{equation*}
  \{ 0,1,\ldots,n-m \} = \{ i + jm : 0 \leq i < m,\ 0 \leq j < q_i \},%
\end{equation*}
and all integers in the set on the right-hand side are uniquely parametrized by $i$ and $j$.%
\end{lemma}

\begin{proof}
It is clear that $(i_1,j_1) \neq (i_2,j_2)$ implies $i_1 + j_1m \neq i_2 + j_2m$, since $0 \leq i_1,i_2 < m$. Hence, it suffices to show that the two sets are equal. To this end, we first show that $i + jm \leq n - m$, whenever $0 \leq i < m$ and $0 \leq j < q_i$. Since $j < q_i$, we have $(j+1)m \leq q_i m + r_i$. Adding $i$ on both sides yields $(i + jm) + m \leq n$, or equivalently $i+jm \leq n - m$.%

Conversely, let us show that every number $l$ between $0$ and $n-m$ can be written as $i + jm$ with $0 \leq i < m$ and $0 \leq j < q_i$. To this end, let $i,j$ be the unique nonnegative integers so that $l = i + jm$ with $0 \leq i < m$. We need to show that $j < q_i$. This is equivalent to%
\begin{equation*}
  l = i + jm < i + q_im = n - r_i.%
\end{equation*}
This inequality holds, because $l < n - (m - 1) \leq n - r_i$, since $0 \leq r_i < m$.%
\end{proof}

The next lemma is taken from \cite[Thm.~4]{Cra}.%

\begin{lemma}\label{lem_crauel}
Consider a measurable skew-product $\Theta_n:\Omega\tm M \rightarrow \Omega \tm M$ ($n\in\Z$), $(\omega,x) \mapsto (\theta^n\omega,\varphi(n,\omega)x)$, where $M$ is a compact metrizable space equipped with its Borel-$\sigma$-algebra and $(\Omega,\FC,P)$ is a probability space. We assume that $\theta$ preserves the measure $P$ and that $\varphi$ is continuous in $x$. Let $(\sigma_n)_{n\in\N}$ be a family of random measures on $\Omega\tm M$, $(t_n)_{n\in\N}$ a sequence of integer-valued random variables such that $t_n \geq n$ almost surely. Put%
\begin{equation*}
  \gamma^n := \frac{1}{t_n}\sum_{k=0}^{t_n-1}(\Theta_k)_*\sigma_n.%
\end{equation*}
Then any limit point $\gamma^*$ of $(\gamma^n)_{n\in\N}$ is $\Theta$-invariant, i.e., $(\Theta_n)_*\gamma^* = \gamma^*$ for all $n\in\Z$.%
\end{lemma}

\subsection{Graph transform and volume estimates}\label{subsec_volume_lemma}%

In this subsection, we prove the volume lemma (Lemma \ref{lem_volume}). This is a straightforward adaptation of the proof given in \cite{You} for autonomous systems. The proof does not use any structure of the base space $\UC$ except for compactness and metrizability. In fact, the proof applies to arbitrary skew-product flows with compact metrizable base space and a compact set $Q$ satisfying the partial hyperbolicity assumption (P2).%

We will use the following notation: If $L$ is a linear map between Euclidean spaces, we write $m(L) = \min_{|x|=1}|Lx|$ for its conorm and $\|L\| = \max_{|x|=1}|Lx|$ for its norm. If $i,j\in\Z$ with $i \leq j$, we write $[i:j] = [i,j] \cap \Z = \{i,i+1,\ldots,j\}$. We simply write $d$ for the product metric on $\UC \tm M$. By $\Lip(\cdot)$ we denote the Lipschitz constant of a map.%

We first extend both $E^+$ and $E^{0-}$ continuously to a compact neighborhood $\KC$ of $L(Q)$, without requiring invariance of these extended bundles. We define the continuous function $\xi:\KC \rightarrow \R$,%
\begin{equation*}
  \xi(u,x) := -\log\bigl|\det\rmD\varphi_{1,u}(x)_{|E^+(u,x)}\bigr|.%
\end{equation*}
Recall that we denote by $\lambda>0$ the expansivity constant on $E^+$ as required in (P2).%

For $r>0$ and $(u,x)\in\KC$ we define%
\begin{align*}
  E^+_r(u,x) := &\left\{v\in E^+(u,x)\ :\ |v| < r \right\}, \quad  E^{0-}_r(u,x) := \left\{ v\in E^{0-}(u,x)\ :\ |v| < r \right\},\\
	& \qquad\quad\quad E_r(u,x) := E^+_r(u,x) \tm E^{0-}_r(u,x).%
\end{align*}

We fix some $0 < \ep \ll \lambda$. According to (P2), there exists $N\in\N$ such that for all $(u,x)\in L(Q)$,%
\begin{align*}
  v \in E^+(u,x) \quad\Rightarrow\quad |\rmD\varphi_{N,u}(x)v| \geq \rme^{\lambda N}|v|,\\
	v \in E^{0-}(u,x) \quad\Rightarrow\quad |\rmD\varphi_{N,u}(x)v| \leq \rme^{\ep N}|v|.%
\end{align*}

We work with a time-discretized system in the following, hence we may assume $N=1$. Thus,%
\begin{equation}\label{eq_normest}
  m(\rmD\varphi_{1,u}(x)_{|E^+(u,x)}) \geq \rme^{\lambda},\quad \|\rmD\varphi_{1,u}(x)_{|E^{0-}(u,x)}\| \leq \rme^{\ep},\quad (u,x) \in L(Q).%
\end{equation}

We fix an $r>0$ small enough so that the following holds: If $(u,x) \in \KC$ and $\Phi_1(u,x) \in \KC$, then%
\begin{equation*}
  \widetilde{\varphi}_{u,x}:E_r(u,x) \rightarrow T_{\varphi(1,x,u)}M,\quad \widetilde{\varphi}_{u,x} := \exp_{\varphi(1,x,u)}^{-1} \circ \varphi_{1,u} \circ \exp_x%
\end{equation*}
makes sense. Using our assumption that the dimensions of $E^+$ and $E^{0-}$ are constant, we put%
\begin{equation*}
  p := \dim E^+ \mbox{\quad and \quad } q := \dim E^{0-}.%
\end{equation*}
In the following, we identify $E^+(u,x)$ with $\R^p$ and $E^{0-}(u,x)$ with $\R^q$.%

For the rest of this subsection, we fix%
\begin{equation*}
  0 < \lambda' < \lambda \mbox{\quad and\quad} \ep' := 2\ep.%
\end{equation*}

The idea of the following lemma is to replace $\widetilde{\varphi}_{u,x}$ for $(u,x) \in \KC \backslash L(Q)$ with a linear map $L_{u,x}$ which is $C^{\Lip}$-close to $\widetilde{\varphi}_{u,x}$ in a small neighborhood of $0 \in T_xM$ and preserves $E^+$ and $E^{0-}$.%

\begin{lemma}\label{lem_33}
For every $\tau>0$ there exist $\delta,r_0>0$ such that for all $(u,x) \in \KC$ with the property%
\begin{equation}\label{eq_prop}
  \exists (v,y) \in L(Q) \mbox{ with } d((u,x),(v,y)) \leq \delta \mbox{ and } d(\Phi_1(u,x),\Phi_1(v,y)) \leq \delta%
\end{equation}
it holds that%
\begin{equation*}
  \Lip(\widetilde{\varphi}_{u,x} - L_{u,x})_{|E_{r_0}(u,x)} < \tau,%
\end{equation*}
where $L_{u,x} \in \End(\R^{p+q})$ is a linear map of the form $L_{u,x} = L_{u,x}^1 \oplus L_{u,x}^2$ with $L_{u,x}^1\in\End(\R^p)$, $L_{u,x}^2\in\End(\R^q)$ and%
\begin{equation*}
  m(L_{u,x}^1) \geq \rme^{\lambda'},\quad \|L_{u,x}^2\| \leq \rme^{\ep'}.%
\end{equation*}
\end{lemma}

\begin{proof}
For the given $\tau>0$ we choose $r_0' < r$ small enough so that%
\begin{equation}\label{eq_zerolipest}
  \Lip(\widetilde{\varphi}_{v,y} - \rmD\widetilde{\varphi}_{v,y}(0))_{|E_{r_0'}(v,y)} < \frac{\tau}{3} \mbox{\quad for all\ } (v,y) \in L(Q).%
\end{equation}
Existence of such $r_0'$ can be seen as follows. Writing%
\begin{equation*}
  r_{v,y}(w) := \widetilde{\varphi}_{v,y}(w) - \rmD\widetilde{\varphi}_{v,y}(0)w,%
\end{equation*}
we find that for $w_1,w_2 \in E_r(v,y)$,%
\begin{equation*}
  |r_{v,y}(w_1) - r_{v,y}(w_2)| \leq \|\rmD r_{v,y}(\xi)\| \cdot |w_1 - w_2|%
\end{equation*}
with a point $\xi \in [w_1,w_2]$. The map $(v,y,\xi) \mapsto \|\rmD r_{v,y}(\xi)\|$ is continuous and satisfies $\|\rmD r_{v,y}(0)\| \equiv 0$. Hence, by uniform continuity on compact sets, we find the desired $r_0'$.%

Now choose $\delta$ small enough so that for any $(u,x)\in\KC$ satisfying \eqref{eq_prop} with some $(v,y)\in L(Q)$ we can find linear isomorphisms $h_1:T_xM \rightarrow T_yM$ and $h_2:T_{\varphi(1,x,u)}M \rightarrow T_{\varphi(1,y,v)}M$ preserving $E^+$ and $E^{0-}$ with $\max\{\Lip(h_1),\Lip(h_2^{-1})\} < \sqrt{2}$. This is possible by uniform continuity of $E^+(\cdot,\cdot)$ and $E^{0-}(\cdot,\cdot)$ on the compact set $L(Q)$. Define%
\begin{equation*}
  \widehat{\varphi}_{(u,x)\rightarrow(v,y)} := h_2^{-1} \circ \widetilde{\varphi}_{y,v} \circ h_1,\quad L_{u,x} := \rmD\widehat{\varphi}_{(u,x)\rightarrow(v,y)}(0).%
\end{equation*}
We have%
\begin{equation*}
  \widehat{\varphi}_{(u,x)\rightarrow(v,y)} - \rmD\widehat{\varphi}_{(u,x)\rightarrow(v,y)}(0) = h_2^{-1} \circ (\widetilde{\varphi}_{v,y} - \rmD\widetilde{\varphi}_{(v,y)}(0)) \circ h_1.%
\end{equation*}
Hence, using \eqref{eq_zerolipest}, we can find $0 < r_0 < r_0'$ such that%
\begin{equation}\label{eq_firstlipest}
  \Lip\left(\widehat{\varphi}_{(u,x)\rightarrow(v,y)} - \rmD\widehat{\varphi}_{(u,x)\rightarrow(v,y)}(0)\right)_{|E_{r_0}(u,x)} < \frac{2\tau}{3}%
\end{equation}
whenever $d((u,x),(v,y)) \leq \delta$ and $d(\Phi_1(u,x),\Phi_1(v,y)) \leq \delta$.%

The linear map $L_{u,x}$ is of the form $L_{u,x}^1 \oplus L_{u,x}^2$ with respect to $E^+(u,x) \oplus E^{0-}(u,x)$ and satisfies%
\begin{equation*}
  L_{u,x} = h_2^{-1} \circ \rmD\widetilde{\varphi}_{v,y}(0) \circ h_1.%
\end{equation*}
Hence, if $\delta$ is small enough, we can choose the Lipschitz constants of $h_1$ and $h_2^{-1}$ small enough to obtain $m(L_{u,x}^1) \geq \rme^{\lambda'}$ and $\|L_{u,x}^2\| \leq \rme^{\ep'}$, using \eqref{eq_normest}. Finally, for $\delta$ small enough, $d((u,x),(v,y)) \leq \delta$ and $d(\Phi_1(u,x),\Phi_1(v,y)) \leq \delta$ imply%
\begin{equation}\label{eq_secondlipest}
  \Lip (\widetilde{\varphi}_{u,x} - \widehat{\varphi}_{(u,x)\rightarrow(v,y)})_{|E_{r_0}(u,x)} < \frac{\tau}{3},%
\end{equation}
since $E^+$ and $E^{0-}$ and also $\rmD\widetilde{\varphi}_{\cdot,\cdot}(\cdot)$ are continuous. We end up with%
\begin{align*}
  \Lip ( \widetilde{\varphi}_{u,x} - L_{u,x} ) &\leq \Lip ( \widetilde{\varphi}_{u,x} - \widehat{\varphi}_{(u,x)\rightarrow(v,y)} )_{|E_{r_0}}\\
	& \qquad + \Lip (\widehat{\varphi}_{(u,x)\rightarrow(v,y)} - \rmD\widehat{\varphi}_{(u,x)\rightarrow(v,y)}(0) )_{|E_{r_0}} \stackrel{\eqref{eq_firstlipest},\eqref{eq_secondlipest}}{<} \frac{\tau}{3} + \frac{2\tau}{3} = \tau.%
\end{align*}
The proof is complete.%
\end{proof}

Now we introduce some notation to describe the graph transform.%

We write $R^p(r) := \{x\in\R^p : |x| < r\}$ and $R(r) := R^p(r) \tm R^q(r)$. Let $F:R(r) \rightarrow \R^p \tm \R^q$ be a map. The graph transform by $F$ from $R(r)$ to $R(s)$, for which we write $\Gamma = \Gamma_{r,s}(F)$, is defined as follows. Let $g:R^p(r) \rightarrow R^q(r)$. We say that $\Gamma g$ is well-defined if it is a function from $R^p(s)$ to $R^q(s)$ satisfying%
\begin{equation}\label{eq_graphtransform}
  \graph(\Gamma g) = F(\mbox{graph}(g)) \cap R(s).%
\end{equation}
If $g$ is a $C^1$-function, then $\graph(g)$ can be viewed as a submanifold of $R(r)$ and we write $T(\graph(g))$ for its tangent bundle.%

\begin{lemma}\label{lem_34}
There exist $\tau,\sigma>0$ (depending on $p,q,\lambda',\ep'$) with the following property. Let $r>0$ be arbitrary and let $F:R(r) \rightarrow \R^p \tm \R^q$ be a $C^1$-diffeomorphism onto its image with $F(0) = 0$ and $\widehat{L} \in \End(\R^{p+q})$ of the form $\widehat{L} = \widehat{L}_1 \oplus \widehat{L}_2$ with $\widehat{L}_1 \in \End(\R^p)$, $\widehat{L}_2 \in \End(\R^q)$ such that the following conditions are satisfied:%
\begin{equation*}
  \Lip(F - \widehat{L}) < \tau,\quad m(\widehat{L}_1) \geq \rme^{\ep'}, \quad \|\widehat{L}_2\| \leq \rme^{\lambda'}.%
\end{equation*}
If $g:R^p(r) \rightarrow R^q(r)$ is a $C^1$-map with $|g(0)| \leq r/2$ and $\|\rmD g\| \leq \sigma$,\footnote{By $\|\rmD g\| \leq \sigma$, we mean $\|\rmD g(z)\| \leq \sigma$ for all $z$ in the domain of $g$.} then $\Gamma g = \Gamma_{r,r\rme^{\ep'}}(F)g$ is well-defined with%
\begin{equation*}
  |\Gamma g(0)| \leq \frac{r}{2}\rme^{2\ep'},\quad \|\rmD \Gamma g\| \leq \sigma,\quad \rme^{-2\ep'} \leq \frac{|\det \rmD F_{|T(\graph(g))}|}{|\det \widehat{L}_1|} \leq \rme^{2\ep'}.%
\end{equation*}
\end{lemma}

\begin{proof}
In this proof, $h$ will always denote an element of $\Hom(\R^p,\R^q)$. We first choose $\tau_1>0$ such that the following holds: For every $L \in \End(\R^{p+q})$, if $\|L - \widehat{L}\| \leq \tau_1$ and $\|h\| \leq 1$, then%
\begin{equation}\label{eq_firstdil}
  \rme^{-\ep'} \leq \frac{|\det L_{|\graph(h)}|}{|\det\widehat{L}_{|\graph(h)}|} \leq \rme^{\ep'}.%
\end{equation}
We choose $\sigma>0$ such that $\|h\| \leq \sigma$ implies%
\begin{equation}\label{eq_seconddil}
  \rme^{-\ep'} \leq \frac{|\det \widehat{L}_{|\graph(h)}|}{|\det \widehat{L}_1|} \leq \rme^{\ep'}.%
\end{equation}
It follows from standard graph transform estimates that given $\sigma,\lambda',\ep'$, there exists $0 < \tau \leq \tau_1$ such that $\Lip(F - \widehat{L}) < \tau$ implies that $\Gamma g$ exists and satisfies $|\Gamma g(0)| \leq (r/2)\rme^{2\ep'}$, $\|\rmD\Gamma g\| \leq \sigma$.%

Let $y = (z,g(z)) \in \graph(g)$, $L := \rmD F(y)$ and $h := \rmD g(z)$. Then the following holds:%
\begin{enumerate}
\item[(i)] $T_y\graph(g) = \graph(h)$ (easy to see).%
\item[(ii)] $\|L - \widehat{L}\| = \|\rmD F(y) - \widehat{L}\| \leq \Lip(F - \widehat{L}) < \tau \leq \tau_1$.%
\item[(iii)] $\|h\| = \|\rmD g(z)\| \leq \sigma$ by assumption.%
\end{enumerate}
Since%
\begin{equation*}
  \frac{|\det \rmD F(y)_{|T_y(\graph(g))}|}{|\det \widehat{L}_1|} = \frac{|\det L_{|\graph(h)}|}{|\det \widehat{L}_{|\graph(h)}|} \cdot \frac{|\det \widehat{L}_{|\graph(h)}|}{|\det \widehat{L}_1|},%
\end{equation*}
the inequalities \eqref{eq_firstdil} and \eqref{eq_seconddil} yield%
\begin{equation*}
  \rme^{-2\ep'} \leq \frac{|\det \rmD F(y)_{|T_y(\graph(g))}|}{|\det \widehat{L}_1|} \leq \rme^{2\ep'},%
\end{equation*}
completing the proof.%
\end{proof}

We will apply the above lemma to $F := \widetilde{\varphi}_{u,x}$ with $(u,x) \in \KC$ and $\widehat{L} = L_{u,x}$ as in Lemma \ref{lem_33}. In the following, let%
\begin{equation*}
  B^{n,u}_{\rho}(x) := \left\{ y\in M\ :\ d(\varphi(j,x,u),\varphi(j,y,u)) \leq \rho,\ 0 \leq j \leq n \right\},%
\end{equation*}
which slightly differs from \eqref{eq_def_bowenball} without consequences on the validity of the volume lemma, however.%

The following proposition is the adaptation of \cite[Lem.~2]{You} to our situation.%

\begin{proposition}\label{prop_volest}
There exist a compact neighborhood $\KC$ of $L(Q)$ (possibly smaller than the original $\KC$) and a constant $\rho_0>0$ such that for each $0 < \rho \leq \rho_0$ there is a constant $C_{\rho} \geq 1$ such that%
\begin{equation*}
  C_{\rho}^{-1} \rme^{-4\ep'n} \rme^{\sum_{i=0}^n\xi(\Phi_i(u,x))} \leq \vol(B^{n,u}_{\rho}(x)) \leq C_{\rho} \rme^{4\ep' n}\rme^{\sum_{i=0}^n\xi(\Phi_i(u,x))}%
\end{equation*}
for all $n\in\N$ and $(u,x)\in\KC_n := \bigcap_{i=0}^n\Phi_{-i}(\KC)$.%
\end{proposition}

\begin{proof}
Let $\tau,\sigma>0$ be given by Lemma \ref{lem_34} and $\delta,r_0>0$ (depending on $\tau$) by Lemma \ref{lem_33}. Since $\sigma$ can be chosen arbitrarily small in Lemma \ref{lem_34}, we may assume that $16\sigma < 1$. Let%
\begin{align*}
  \KC := \bigl\{ (u,x) \in \UC \tm M\ :\ &\exists (v,y) \in L(Q) \mbox{ s.t.}\\
	               & \max\{d((u,x),(v,y)),d(\Phi_1(u,x),\Phi_1(v,y))\} \leq \delta \bigr\},%
\end{align*}
which we can assume to be contained in the original $\KC$ (where $E^+$ and $E^{0-}$ are defined).%

For all sufficiently small $\rho$ and $(u,x) \in \KC_n$ we have the inclusion%
\begin{align*}
  B^{n,u}_{\rho}(x) &\subset \exp_x\left\{ \xi\in R^p(r_0) \tm R^q\Bigl(\frac{r_0}{4}\Bigr)\ :\ \forall j \in [0:n],\ \widetilde{\varphi}_{u,x}^j(\xi) \in R^p(r_0) \tm R^q\Bigl(\frac{r_0}{4}\Bigr) \right\}\\
										&=: \exp_x\left( N^{n,u}_{r_0}(x) \right),%
\end{align*}
where $\widetilde{\varphi}_{u,x}^j = \widetilde{\varphi}_{\Phi_{j-1}(u,x)} \circ \cdots \circ \widetilde{\varphi}_{u,x}$ (cf.~\cite[Lem.~4.2]{DK1} for a detailed argument).

For $w\in R^q(r_0)$, denote by $g^w:R^p(r_0) \rightarrow R^q(r_0)$ the constant function equal to $w$. Let%
\begin{equation*}
  A := \left\{ w\in R^q(r_0/4)\ :\ \forall j \in [0:n],\ \graph(\Gamma^jg^w) \subset R^p(r_0) \tm R^q(r_0/2) \right\},%
\end{equation*}
where $\Gamma^j = \Gamma_j \circ \cdots \circ \Gamma_1$ and $\Gamma_j = \Gamma_{r_0,r_0\rme^{\ep'}}(\widetilde{\varphi}_{\Phi_j(u,x)})$. We claim that%
\begin{equation*}
  N^{n,u}_{r_0}(x) \subset \bigcup_{w\in A} (\widetilde{\varphi}_{u,x}^n)^{-1}\graph(\Gamma^ng^w).%
\end{equation*}
Let $\xi = (\xi_p,\xi_q) \in N^{n,u}_{r_0}(x)$ and $w := \xi_q \in R^q(r_0/4)$. We have $|g^w(0)| \leq r_0/2$ and $\|\rmD g^w\| = 0 < \sigma$, so by Lemma \ref{lem_34}, $\Gamma g^w$ is well-defined with $\|\rmD\Gamma g^w\| \leq \sigma$.%

As $\xi \in N^{n,u}_{r_0}(x)$, $\widetilde{\varphi}_{u,x}(\xi) \in R^p(r_0) \tm R^q(r_0/4)$ implying $|\Gamma g^w(0)| \leq r_0/4$ (from \eqref{eq_graphtransform}). As $\|\rmD \Gamma g^w\| \leq \sigma < 1/8$ (by the choice of $\sigma$), on $R^p(r_0)$ we have%
\begin{equation*}
  |\Gamma g^w| \leq |\Gamma g^w(0)| + \|\rmD \Gamma g^w\| \cdot r_0 \leq \frac{r_0}{4} + \frac{r_0}{8} < \frac{r_0}{2}.%
\end{equation*}
Hence, the image of $\Gamma g^w$, restricted to $R^p(r_0)$, is contained in $R^q(r_0)$. We conclude that $\Gamma g^w$ satisfies all the assumptions of Lemma \ref{lem_34}, implying that $\Gamma^2 g^w$ is well-defined with $\|\rmD \Gamma^2 g^w\| \leq \sigma$.%

Iterating this process, we obtain that for every $i \in [1:n]$, $\Gamma^ig^w$ is well-defined and $|\Gamma^ig^w| < r_0/2$, i.e., $\graph(\Gamma^ig^w) \subset R^p(r_0) \tm R^q(r_0/2)$. So $w\in A$ and $\xi \in (\widetilde{\varphi}_{u,x}^n)^{-1}\graph(\Gamma^ng^w)$. (Observe that every time we apply the graph transform, $r_0$ gets multiplied with $\rme^{\ep'}$. However, every time we can restrict the domain of the resulting function again to $R^p(r_0)$.)%

Let $m_p$ denote the $p$-dimensional Riemannian measure on any $p$-dimensional submanifold of a tangent space of $M$. Then%
\begin{align*}
  m_p\left(\left(\widetilde{\varphi}_{u,x}^n\right)^{-1}\graph(\Gamma^ng^w)\right) &\leq \max_{w\in A,\ z \in \graph g^w}\left|\det(\rmD\widetilde{\varphi}_{u,x}^n)_{|T_z\graph(g^w)}\right|^{-1} \cdot m_p(\graph(\Gamma^ng^w))\\
	&\leq \rme^{2\ep'n}\prod_{i=0}^{n-1}\bigl|\det L_{\Phi_i(u,x)}^1\bigr|^{-1} \cdot m_p(\graph(\Gamma^ng^w))\allowdisplaybreaks\\
	&\leq \rme^{2\ep'n}\prod_{i=0}^{n-1}\bigl|\det L_{\Phi_i(u,x)}^1\bigr|^{-1} \cdot Cr_0^p\allowdisplaybreaks\\
	&\leq Cr_0^p \rme^{4\ep'n}\prod_{i=0}^{n-1}\Bigl|\det(\rmD\widetilde{\varphi}_{\Phi_i(u,x)})(0)_{|\R^p}\Bigr|^{-1}\allowdisplaybreaks\\
	   &= Cr_0^p \rme^{4\ep'n}\prod_{i=0}^{n-1}\left|\det\rmD\varphi_{1,\theta_iu}(\varphi(i,x,u))_{|E^+(\Phi_i(u,x))}\right|^{-1}\allowdisplaybreaks\\
		 &= Cr_0^p \rme^{4\ep'n}\prod_{i=0}^{n-1}\rme^{\xi(\Phi_i(u,x))} = Cr_0^p \rme^{4\ep'n}\rme^{\sum_{i=0}^{n-1}\xi(\Phi_i(u,x))}.%
\end{align*}
Here the second and the fourth line follow from Lemma \ref{lem_34} (in the second case, the lemma is applied with $g = 0$). Observe that the sets $\left(\widetilde{\varphi}_{u,x}^n\right)^{-1}\graph(\Gamma^ng^w)$, $w\in A$, are pairwisely disjoint. Integrating over $A$ yields%
\begin{equation*}
  m_{p+q}\left(N^{n,u}_{r_0}(x)\right) \leq C'\rme^{4\ep'n}\prod_{i=0}^{n-1}\rme^{\xi(\Phi_i(u,x))}%
\end{equation*}
for some constant $C'>0$ depending on $r_0$ (here we use that the angle between the subspace $E^+$ and $E^{0-}$ is bounded on $\KC$). Since the application of $\exp_x$ results in a volume distortion, which is bounded over $(u,x) \in \KC$, and $B^{n,u}_{\rho}(x) \subset \exp_x(N^{n,u}_{r_0}(x))$, the upper estimate follows. Analogous arguments yield the lower estimate (cf.~\cite[Lem.~2]{You}).%
\end{proof}

Now we can conclude the proof of the volume lemma.%

\begin{proof} (of Lemma \ref{lem_volume})
We put $\zeta := 4\ep' = 8\ep$. Let $\eta>0$ be chosen such that $\KC$ in Proposition \ref{prop_volest} contains the closed $\eta$-neighborhood of $L(Q)$. Let $n\in\N$ and $x\in Q_{\rmd}(n+1,u,\eta)$ for some $u\in\UC$, i.e., $\dist(\varphi(i,x,u),Q(\theta_iu)) \leq \eta$ for $0 \leq i \leq n$. For each $0\leq i \leq n$ choose $y_i \in M$ with $(\theta_iu,y_i) \in L(Q)$ and $d(\varphi(i,x,u),y_i)\leq \eta$. This yields $\Phi_i(u,x) \in \KC$ for $0 \leq i \leq n$, so Proposition \ref{prop_volest} implies%
\begin{equation*}
  C_{\rho}^{-1}\rme^{-4\ep'n}\rme^{\sum_{i=0}^{n-1}\xi(\Phi_i(u,x))} \leq \vol(B^{n,u}_{\rho}(x)) \leq C_{\rho}\rme^{4\ep'n}\rme^{\sum_{i=0}^{n-1}\xi(\Phi_i(u,x))}%
\end{equation*}
for $0 < \rho \leq \rho_0$. Define the desired extension of $J^+\varphi_{n,u}(x)$ to $\KC$ by%
\begin{equation*}
  J^+\varphi_{n,u}(x) := \prod_{i=0}^{n-1} \left|\det \rmD\varphi_{1,\theta_iu}(\varphi(i,x,u))_{|E^+(\Phi_i(u,x))}\right| = \rme^{-\sum_{i=0}^{n-1}\xi(\Phi_i(u,x))}.%
\end{equation*}
This extension satisfies multiplicativity by definition and thus the volume estimates are proved.\footnote{The number $\ep$ in the formulation of Lemma \ref{lem_volume}, of course, here is $\eta$, and $\delta_0$ here is $\rho_0$.} As the proof shows, $\ep$ can be chosen arbitrarily small. By Lemma \ref{lem_34} then also $\tau$ and $\sigma$ must become small, and by Lemma \ref{lem_33} also $\delta$ and $\rho_0$.%
\end{proof}

\section{Acknowledgements}%
  
We express our gratitude to Eduardo Garibaldi and Ian Morris for helping us to gain a better understanding of subadditive cocycles. We also thank Maxence Novel for helping us to gain a better understanding of the volume lemma. The first author was supported by FAPESP Grant 2016/11135-2 and a Guest Scientist Grant from the University of Passau, where part of this work was done.%


\begin{thebibliography}{99}
\bibitem{Arn} L.~Arnold. Random Dynamical Systems. Springer Monographs in Mathematics. Springer-Verlag, Berlin, 1998.%
\bibitem{ADS} V.~Ayala, A.~Da Silva, L.~A.~B.~San Martin. Control systems on flag manifolds and their chain control sets. Discrete Contin.\ Dyn.\ Syst.\ 37 (2017), no.\ 5, 2301--2313.%
\bibitem{Bog} T.~Bogensch\"{u}tz. Entropy, pressure, and a variational principle for random dynamical systems. Random Comput.\ Dynam.\ 1 (1992/93), no.\ 1, 99--116.%
\bibitem{SM1} C.~J.~Braga Barros, L.~A.~B.~San Martin. Chain transitive sets for flows on flag bundles. Forum Math.\ 19 (2007), 1, 19--60.%
\bibitem{CZh} Y.~Y.~Chen, Y.~Zhao. Ergodic optimization for a sequence of continuous functions. (Chinese) Chinese Ann.\ Math.\ Ser.\ A 34 (2013), no.\ 5, 589--598; translation in Chinese J.\ Contemp.\ Math.\ 34 (2013), no.\ 4, 351--360.%
\bibitem{Col} F.~Colonius. Minimal bit rates and entropy for exponential stabilization. SIAM J.\ Control Optim.\ 50 (2012), 2988--3010.%
\bibitem{Co2} F.~Colonius. Metric invariance entropy and conditionally invariant measures. Ergod.\ Th.\ \& Dynam.\ Sys. First published online: 20 Oct.~2016. doi: 10.1017/etds.2016.72.%
\bibitem{CDu} F.~Colonius, W.~Du. Hyperbolic control sets and chain control sets. J.\ Dynam.\ Control Systems 7 (2001), no.\ 1, 49--59.%
\bibitem{CKa} F.~Colonius, C.~Kawan. Invariance entropy for control systems. SIAM J.\ Control Optim.\ 48 (2009), no.\ 3, 1701--1721.%
\bibitem{CKl} F.~Colonius, W.~Kliemann. The Dynamics of Control. Birkh\"{a}user Boston, Inc., Boston, MA, 2000.%
\bibitem{Cra} H.~Crauel. Lyapunov exponents and invariant measures of stochastic systems on manifolds. Lyapunov exponents (Bremen, 1984), 271--291, Lecture Notes in Math., 1186, Springer, Berlin, 1986.%
\bibitem{DK1} A.~Da Silva, C.~Kawan. Invariance entropy of hyperbolic control sets. Discrete Contin.\ Dyn.\ Syst.\ 36 (2016), no.\ 1, 97--136.%
\bibitem{DK2} A.~Da Silva, C.~Kawan. Hyperbolic chain control sets on flag manifolds. J.\ Dyn.\ Control Syst.\ 22, no.\ 4, 725--745 (2016).%
\bibitem{DK3} A.~Da Silva, C.~Kawan. Robustness of critical bit rates for practical stabilization of networked control systems. Submitted, Oct.~2016.%
\bibitem{DK4} A.~Da Silva, C.~Kawan. Lyapunov exponents and partial hyperbolicity of chain control sets on flag manifolds. Submitted, 2018; Preprint: arXiv:1803.05292%
\bibitem{DVa} A.~Diwadkar, U.~Vaidya. Limitations for nonlinear observation over erasure channel. IEEE Trans.\ Automat.\ Control 58 (2013), no.\ 2, 454--459.%
\bibitem{EMa} M.~I.~El-Hawwary, M.~Maggiore. Reduction principles and the stabilization of closed sets for passive systems. IEEE Trans.\ Automat.\ Control 55 (2010), no.\ 4, 982--987.%
\bibitem{FHu} D.--J.~Feng, W.~Huang. Lyapunov spectrum of asymptotically sub-additive potentials. Comm.\ Math.\ Phys.\ 297 (2010), no.\ 1, 1--43.%
\bibitem{Gru} L.~Gr\"{u}ne. A uniform exponential spectrum for linear flows on vector bundles. J.\ Dynam.\ Differential Equations 12 (2000), no.\ 2, 435--448.%
\bibitem{HKa} B.~Hasselblatt, A.~Katok (Eds.). Handbook of Dynamical Systems. Elsevier, 2002.%
\bibitem{KHa} A.~Katok, B.~Hasselblatt. Introduction to the Modern Theory of Dynamical Systems. Cambridge University Press, Cambridge, 1995.%
\bibitem{Kaw} C.~Kawan. Invariance Entropy for Deterministic Control Systems. An Introduction. Lecture Notes in Mathematics, 2089. Springer, 2013.%
\bibitem{Ka2} C.~Kawan. On the structure of uniformly hyperbolic chain control sets. Systems Control Lett.\ 90, 71--75 (2016).%
\bibitem{Ka3} C.~Kawan. Uniformly hyperbolic control theory. Annual Reviews in Control 44, 89--96, 2017.%
\bibitem{LHe} D.~Liberzon, J.~P.~Hespanha. Stabilization of nonlinear systems with limited information feedback. IEEE Trans.\ Automat.\ Control 50 (2005), no.\ 6, 910--915.%
\bibitem{Liu} P.--D.~Liu. Random perturbations of Axiom A basic sets. J.\ Statist.\ Phys.\ 90 (1998), no.\ 1--2, 467--490.%
\bibitem{Ma1} A.--R.~Mansouri. Topological obstructions to submanifold stabilization. IEEE Trans.\ Automat.\ Control 55 (2010), no.\ 7, 1701--1703.%
\bibitem{Ma2} A.--R.~Mansouri. Local asymptotic feedback stabilization to a submanifold: topological conditions. Systems Control Lett.\ 56 (2007), no.\ 7--8, 525--528.%
\bibitem{MPo} A.~S.~Matveev, A.~Pogromsky. Observation of nonlinear systems via finite capacity channels: Constructive data rate limits. Automatica 70 (2016), 217--229.%
\bibitem{MAl} J.~M.~Montenbruck, F.~Allg\"{o}wer. Asymptotic stabilization of submanifolds embedded in Riemannian manifolds. Automatica J.\ IFAC 74 (2016), 349--359.%
\bibitem{Mor} I.~Morris. Mather sets for sequences of matrices and applications to the study of joint spectral radii. Proc.\ Lond.\ Math.\ Soc. (3) 107 (2013), no.\ 1, 121--150.%
\bibitem{Nea} G.~N.~Nair, R.~J.~Evans, I.~M.~Y.~Mareels, W.~Moran. Topological feedback entropy and nonlinear stabilization. IEEE Trans.\ Automat.\ Control 49 (2004), 9, 1585--1597.%
\bibitem{Ne2} G.~N.~Nair, F.~Fagnani, S.~Zampieri, R.~J.~Evans. Feedback control under data rate constraints: An overview. Proceedings of the IEEE 95.1 (2007), 108--137.%
\bibitem{SM2} L.~A.~B.~San Martin. Order and domains of attraction of control sets in flag manifolds. J.\ Lie Theory 8 (1998), 2, 335--350.%
\bibitem{SM3} L.~A.~B.~San Martin, L.~Seco. Morse and Lyapunov spectra and dynamics on flag bundles. Ergod.\ Th.\ \& Dynam.\ Sys.\ 30 (2010), 3, 893--922.%
\bibitem{SM4} L.~A.~B.~San Martin, P.~A.~Tonelli. Semigroup actions on homogeneous spaces. Semigroup Forum 50 (1995), 1, 59--88.%
\bibitem{OGY} E.~Ott, C.~Grebogi, J.~A.~Yorke. Controlling chaos. Phys.\ Rev.\ Lett.\ 64 (1990), no.\ 11, 1196--1199.%
\bibitem{Sav} A.~V.~Savkin. Analysis and synthesis of networked control systems: topological entropy, observability, robustness and optimal control. Automatica J.\ IFAC 42 (2006), no.\ 1, 51--62.%
\bibitem{You} L.--S.~Young. Some large deviation results for dynamical systems. Trans.\ Amer.\ Math.\ Soc. 318 (1990), no.\ 2, 525--543.%
\bibitem{Yuk} S.~Y\"{u}ksel. Stationary and ergodic properties of stochastic non-linear systems controlled over communication channels. SIAM J.\ Control Optim.\ 54 (2016), no.\ 5, 2844--2871.%
\bibitem{YBa} S.~Y\"{u}ksel, T.~Ba\c{s}ar. Stochastic Networked Control Systems: Stabilization and Optimization under Information Constraints. Systems \& Control: Foundations \& Applications. Birkh\"{a}user/Springer, New York, 2013.%
\bibitem{ZCa} Y.~Zhao, Y.~Cao. On the topological pressure of random bundle transformations in sub-additive case. J.\ Math.\ Anal.\ Appl. 342 (2008), no.\ 1, 715--725.%
\end{thebibliography}
\end{document}